\documentclass{article}
\usepackage{lipsum} 

\usepackage{graphicx}
\usepackage[sc]{mathpazo} 
\usepackage[T1]{fontenc} 
\usepackage{microtype} 

\usepackage[hmarginratio=1:1,top=32mm,columnsep=18pt]{geometry} 
\usepackage{multicol} 
\usepackage[hang, small,labelfont=bf,up,textfont=it,up]{caption} 
\usepackage{booktabs} 
\usepackage{float} 

\usepackage{lettrine} 
\usepackage{paralist} 

\usepackage{abstract} 

\usepackage{titlesec} 

\titleformat{\section}[block]{\large\scshape\centering}{\thesection.}{1em}{} 
\titleformat{\subsection}[block]{\large}{\thesubsection.}{1em}{} 

\usepackage{fancyhdr} 
\usepackage{hyperref} 
\usepackage{url}

\usepackage{mathtools,amssymb,amsmath,latexsym}
\usepackage{amsthm}

\usepackage{dsfont}
    \newcommand{\R}{\mathds{R}}

\DeclareMathOperator*{\argmin}{argmin} %
\newtheorem{theorem}{Theorem}[section]
\newtheorem{proposition}[theorem]{Proposition}%
\newtheorem{corollary}[theorem]{Corollary}%

\newtheorem{example}{Example}%
\newtheorem{remark}{Remark}%
\newtheorem{lemma}[theorem]{Lemma}%
\newtheorem{assumption}{Assumption}

\title{\vspace{-15mm}\fontsize{20pt}{10pt}\selectfont\textbf{Static multilevel reverse Stackelberg games: existence and computations of best strategies}}

\author{
\large
\textbf{Seyfe Belete Worku$^{1}$, Birilew Belayneh Tsegaw$^{1}$ \& Semu Mitiku Kassa$^{2}$}\\[4mm] 
\textit{ $^1$ Department of Mathematics}\\[2mm]
\textit{Bahir Dar University, P.O.Box 79, Bahir Dar, Ethiopia}\\[2mm]
\textit{(e-mail: seyfv@yahoo.com, birilewb@yahoo.com). }\\[2mm]
\textit{$^2$ Department of Mathematics and Statistical Sciences}\\[2mm]
\textit{Botswana International University of Science and Technology}\\[2mm]
\textit{P/Bag 16, Palapye, Botswana} 
\textit{(e-mail: kassas@biust.ac.bw).} 
\vspace{-5mm}
}

\date{}

\begin{document}

\maketitle 

\thispagestyle{fancy} 

------------------------------------------------------------------------------------------------------------------------------\\
\textbf{Abstract:}  The multilevel reverse Stackelberg game is considered. In this game, the leader controls the outcome by announcing a strategy as a function of decision variables of the followers to his/her own decision space. Corresponding to the leader's strategy, the player in the next level  presents his/her strategy as a function of decision variables of the remaining players. This procedure is repeated until it is the turn of the bottom level player in the hierarchy, who reacts by determining his/her optimal decision variables. The structure of this game can be adopted in  decentralized multilevel decision making like resource allocation, energy market pricing, problems with hierarchical controls. In this paper conditions for existence and construction of affine leader reverse Stackelberg strategies are developed for such problems.  As an extension to the existing literature, we considered nonconvex sublevel sets of objective functions of followers. Moreover, a method to construct multiple reverse Stackelberg strategies for the leader is also presented.


------------------------------------------------------------------------------------------------------------------------------\\



\section{Introduction}\label{intro}

In control problems  where decision making is characterized by a natural hierarchy, sequential control approach can be adopted \cite{scattolini2009architectures}.  Also in model predictive controls of complex dynamical systems composed of subsystems at different layers, hierarchical control approach is essential \cite{scattolini2007hierarchical}.  In order to deal with such problems, a leader-follower solution concept which was introduced by Stackelberg \cite{von1934marktform} can be used as the framework for the resulting optimization problem. In a Stackelberg  game some decision makers are able to act prior to other players which then reacts in a rational manner, i.e. the players in Stackelberg games act in a specific order \cite{bacsar1998dynamic}. Thus, unlike Nash games where players are assumed to act simultaneously, Stackelberg games introduce a sequence of decisions between the players to characterize equilibria. Stackelberg games can be extended to multilevel games in which players  are  distributed throughout a multilevel hierarchy  \cite{Abay_Semu_2013,Abay_Semu_2016,Abay_Semu_2017,Semu_2018}.

Multilevel games are subsets of multilevel hierarchical decision problems  that deal with decentralized decision problems involving interacting players that are distributed throughout a $n$-level hierarchy, $ n\geq 2 $.	The players make their individual decisions in a sequential order, from the top $ 1^{\textrm{st}} $-level leader to the $ 2^{\textrm{nd}}$-level player  up to the bottom, $n^{\textrm{th}}$-level, player with the aim of optimizing their respective objectives. However, the class of reverse Stackelberg strategy is a solution approach where the leader formulates a strategy as a mapping from decision spaces of the followers towards his/her decision space that makes followers to behave as desired \cite{basar1979new,groot2012existence,zheng1982existence}.  Such strategy of the leader induces each of the followers to behave cooperatively in achieving team optimal solution of the next successive level player which eventually coincides with the desired solution of the leader. This strategy is also referred to as equilibrium solution \cite{bacsar1981equilibrium,basar1981stochastic}, incentive strategy \cite{ehtamo1989incentive,zheng1982existence}, reverse Stackelberg strategy  \cite{ho1981information}, inverse Stackelberg strategy \cite{olsder2009phenomena}.

Several researchers have investigated the existence and construction of reverse Stackelberg strategies for bi-level  games  \cite{groot2012existence,groot2016optimal,ho1981information,zheng1982existence,zheng1984stackelberg}. Multilevel Stackelberg games have been  applied in areas such as resource allocation \cite{Semu_2018,mitiku2007multilevel}, electricity pricing \cite{stankova2009stackelberg}, marketing channel \cite{groot2017hierarchical} and road pricing \cite{groot2014toward}.  Existence and construction of reverse Stackelberg strategies for trilevel games have been proved  in \cite{bacsar1981equilibrium,basar1981stochastic,basar1980performance} for the case where the objective functions of all the followers are quadratically convex under dynamic information, and in \cite{mizukami1989constructions} where the objectives are  strictly convex.  However, the strict convexity assumption is a strong condition which may not be satisfied by some practical problems.
Moreover, the works in \cite{bacsar1981equilibrium,basar1981stochastic} provide solutions for linear quadratic cost function structures of players, whereas the one in \cite{mizukami1989constructions} results in only a single reverse Stackelberg strategy for the leader. However, for some practical problems the leader can have infinitely many possible strategies to achieve his/her desired equilibrium. This article presents existence conditions that are applicable to a more general game setting and formulates a solution method that enables to generate multiple strategies.

In this article, we consider static multilevel reverse Stackelberg games and investigate existence of optimal reverse Stackelberg strategies for each of the players in the game. Existence of affine reverse Stackelberg strategy of the leader (and middle-level players) is established under some mild conditions on the  objective functions of the followers at the desired equilibrium.  The existence of optimal affine reverse Stackelberg strategies are presented under more relaxed conditions  than those in \cite{bacsar1981equilibrium,basar1981stochastic,basar1980performance,mizukami1989constructions}. Here, sublevel sets of objective functions of the followers at the desired equilibrium point are required to be connected which is a relatively mild condition in comparison to quadratically convex and strictly convex objective functions in the prior works.

The other main contribution of this work is that it proposes existence and construction  of   an infinite number of affine strategies that can induce the desired behavior of followers.  Such multiple optimal affine reverse Stackelberg strategies for the leader are characterized and constructed using free parameters. The construction of multiple strategies for multilevel reverse Stackelberg games enables  consideration of additional optimization criterion as a secondary objective \cite{cansever1982minimum}. This article also sheds light to the study of constrained version of the problem under consideration.

The paper is organized as follows. The formulation of  multilevel reverse Stackelberg game and definition of a corresponding strategy is included in Section 2. In Section 3, we present the first contribution of this article: the existence conditions for optimal affine reverse Stackelberg  strategy that enables  the leader to achieve his/her desired solution. The cases for convex sublevel sets and nonconvex sublevel sets are dealt with in Subsection 3.1 and Subsection 3.2 respectively.
The first part of Section 4 contains a method to construct only one  optimal affine reverse Stackelberg strategy,  with examples, for the leader. The second part starts with an example on trilevel game having infinitely many leader's reverse Stackelberg strategies. In this section, we present the second contribution of this article: Construction of multiple optimal affine reverse Stackelberg strategies of the leader. The insight brought up by the characterization of multiple optimal reverse Stackelberg strategies towards further study of the constrained is given in Section 5. Finally, the paper is concluded in Section 6 where we discuss concluding remarks and possible future  works.

\section{Preliminaries and Problem Formulation}\label{sec:prem}

In this Section we shall formulate the problem structure of static multilevel reverse Stackelberg games and present their properties. To make the presentation clear and easier to follow, we use a 3-level hierarchical static game as our main problem. However, it can be easily extended to any $n$-level hierarchical static game with one decision maker at each level of the hierarchy by repeating the procedure for the middle level decision maker any finite number of times sequentially.  That means, we use the structure of trilevel games for development of existence theorems as well as solution procedures in the subsequent sections in order to address a general $n$-level static game.

Consider a three-player static game  with three levels of hierarchy. Let the objective functions of the top level player (leader), the middle level player, and the third level player (follower)  be 	$ J_{1}(u^{1},u^{2},u^{3}),J_{2}(u^{1},u^{2},u^{3})$ and $J_{3}(u^{1},u^{2},u^{3}) $ respectively.  The variables  $ u^{1}\in \Omega_{1},u^{2}\in \Omega_{2},u^{3}\in \Omega_{3} $ are decision vectors of the leader, the middle level player and the follower respectively. The strategic representations of objective  functions  of the leader, the middle level player and the follower are	 $J_{1}(\gamma^{1},\gamma^{2},u^{3}),J_{2}(\gamma^{1},\gamma^{2},u^{3})$ and $J_{3}(\gamma^{1},\gamma^{2},u^{3}) $ respectively. The corresponding  strategies $ \gamma^{1} $ and $\gamma^{2}$ belong to admissible class of  strategy spaces $\Gamma^{1}$ and $\Gamma^{2} $ of the leader and the middle level player respectively. For ease of elaboration, we restrict the admissible strategy spaces  $\Gamma^{1}$ and $\Gamma^{2} $ to a class of affine  strategies. We denote the optimal strategies by $ \gamma^{i*}, i= 1,2 $.

Here,  the leader is only interested in controlling the state of the system by enforcing his/her optimal decisions to be adopted by the lower level players in the hierarchy. So the first step in determination of the leader's optimal reverse Stackelberg strategy is to evaluate team (desired) optimal solution of  $ J_{1} $. We assume that $ J_{1} $ has a unique team optimal point so that  the leader  seeks to achieve this unique desired equilibrium. For example,  say $(u^{1d},u^{2d},u^{3d})\in \Omega_{1}\times\Omega_{2}\times\Omega_{3}$ is the global optimum of $ J_{1} $, where $ \Omega_{1},\Omega_{2}$ and $\Omega_{3} $ are respectively decision spaces of the leader, the middle level player and the follower. The problem then becomes for the leader to determine an optimal leader function $ \gamma^{1}: \Omega_{2}\times \Omega_{3} \rightarrow \Omega_{1}  $  that leads to the desired equilibrium. The optimal reverse Stackelberg strategy of the leader $ \gamma^{1*} $ and the corresponding strategy of middle level player $ \gamma^{2*} $ achieves $ (u^{1d},u^{2d},u^{3d}) $ if it satisfies
\begin{align}		
	&( \gamma^{1*}, \gamma^{2*},u^{3d})= \argmin_{(\gamma^{1},\gamma^{2},u^{3})\in \Gamma_{1}\times\Gamma_{2}\times \Omega_{3}}J_{1}(\gamma^{1},\gamma^{2},u^{3}),\label{Ea1}\\	
	&(u^{2d},u^{3d})=\argmin_{(u^{2},u^{3})\in \Omega_{2}\times \Omega_{3}}J_{2}(\gamma^{1*},u^{2},u^{3}),\label{E3}\\
	&	u^{3d}=\argmin_{u^{3}\in \Omega_{3}} J_{3}(\gamma^{1*}, \gamma^{2*},u^{3}),\label{E4}	
\end{align}
where the realization of the strategies are
\begin{align}		
	& \gamma^{1*}(u^{2d},u^{3d}) =u^{1d}, \label{E1}\\	
	&\gamma^{2*}(u^{3d}) =u^{2d}. \label{E2}
\end{align}

If $ \gamma^{1*} $ satisfies Eqs.  (\ref{E3})-(\ref{E1})  and $ \gamma^{2*} $ satisfies Eqs. (\ref{E4}) and (\ref{E2}), then they are called optimal reverse Stackelberg strategies of the leader and the middle level player respectively.

The following proposition presents conditions that should be satisfied by the strategy $ \gamma^{1*} $ so that the leader to achieve the desired solution.
\begin{proposition}
	A fixed optimal strategy $ \gamma^{1*} $ of the leader satisfying Eqs. (\ref{E3}) and  (\ref{E1}) induces the middle level player to form a reverse Stackelberg strategy $ \gamma^{2*} $ to induce the follower.
\end{proposition}

\begin{proof}
	In Eq. (\ref{E3}) we see that the leader's optimal reverse Stackelberg  strategy $\gamma^{1*}$  leads to a team optimal solution of the middle level player as if the follower is cooperating. Thus, the middle level player wants to construct $\gamma^{2*}$ that satisfies Eqs. (\ref{E4}) and (\ref{E2}) 	to achieve his/her desired equilibrium  $ (u^{2d},u^{3d}) $. 	
\end{proof}

\begin{remark}
	In this setting, the leader forces or persuades (using incentives and punishment) the middle level player to choose  $\gamma^{2*}$ that satisfies Eqs. (\ref{E4}) and (\ref{E2}) by the announcement of $\gamma^{1*}$ satisfying Eq. (\ref{E3}). The influence of the leader on $\gamma^{2*}$ can be seen in the relation  $ \gamma^{1*}(u^{2},u^{3})=\gamma^{1*}(\gamma^{2*}(u^{3}),u^{3}) $.
\end{remark}

\section{Existence of  Optimal Affine  Reverse Stackelberg Strategy}\label{sect:existance}
In Subsection 3.1 of this section,  we present conditions under which optimal affine reverse Stackelberg strategy of the leader exists. Existence conditions for leader's reverse Stackelberg strategy in trilevel games under consideration are developed with the assumption that sublevel sets of followers objective functions are convex. This is a more general assumption relative to prior works \cite{bacsar1981equilibrium,basar1981stochastic,basar1980performance,mizukami1989constructions} where cost functions of followers are assumed to be strictly convex. In Subsection 3.2, we further relaxed existence conditions by considering nonconvex sublevel sets.
\subsection{Convex sublevel sets }
In order to prove the existence of affine reverse Stackelberg strategy for the leader, we use geometric properties of convex sublevel sets of objective functions of the middle level player and the follower at the desired solution. Let us define the level sets $ W^{2d} $ and $ W^{3d} $ corresponding to the objective functions of the followers $ J_{2} $ and $ J_{3} $ at the desired optimal solution $ (u^{1d},u^{2d},u^{3d}) $ of the leader as follows.
\begin{equation}\label{E5}
	W^{2d} = \{(u^{1},u^{2},u^{3})\in \Omega_{1}\times \Omega_{2}\times \Omega_{3}: J_{2}(u^{1},u^{2},u^{3})\leq J_{2}(u^{1d},u^{2d},u^{3d})\}.
\end{equation}	
\begin{equation}\label{E10}
	W^{3d} = \{(u^{1},u^{2},u^{3})\in \Omega_{1}\times \Omega_{2}\times \Omega_{3}: J_{3}(u^{1},u^{2},u^{3})\leq J_{3}(u^{1d},u^{2d},u^{3d})\}.
\end{equation}

The leader  should construct the optimal reverse Stackelberg strategy $ \gamma^{1*} $  such that it uniquely intersects with the level set  $ W^{2d} $ at  $ (u^{1d},u^{2d},u^{3d}) $. Under the announced leader's strategy $ \gamma^{1*} $ the middle level player constructs a strategy $ \gamma^{2*} $ such that it uniquely intersects $ 	W^{3d} $ at $ (u^{1d},u^{2d},u^{3d}) $.

We make the following assumptions on decision spaces, sublevel sets and objective functions of the followers.

\begin{assumption}\label{Assm:1}
	The decision spaces $ \Omega_{1}, \Omega_{2}, \Omega_{3} $ are convex and the sublevel sets $ W^{2d} $ and $ W^{3d} $ are connected.
\end{assumption}	
\begin{assumption}\label{Assum:2}
	The class of admissible strategies $ \Gamma^{1} $ and $ \Gamma^{2} $ satisfy
	\begin{align*}
        \Gamma^{1}&=\left\{\gamma^{1} : \Omega_{2}\times\Omega_{3}\rightarrow\Omega_{1}~\vert~~\gamma^{1} \textrm{ is affine and } \gamma^{1}(u^{2d},u^{3d}) = u^{1d}\right\},\\
        \Gamma^{2}&= \left\{\gamma^{2} : \Omega_{3}\rightarrow\Omega_{2}~\vert~~\gamma^{2} \textrm{ is affine and } \gamma^{2}(u^{3d})=u^{2d}\right\}.
	\end{align*}
\end{assumption}

The graph of a function $\gamma^{1}$, denoted by $\text{Graph}(\gamma^{1})$, is a relation that maps points in $\Omega_{2} \times \Omega_{3}$ to $\Omega_{1}$ and is represented by
\[ \text{Graph}(\gamma^{1})= \{(u^{1},u^{2},u^{3}) ~:~ u^{3}\in \Omega_{3}, u^{2} \in \Omega_{2}, u^{1} = \gamma^{1}(u^{2},u^{3})\}.\]

In order for the leader to be able to force the middle level player's decision to his/her desired equilibrium point, the leader should construct his/her affine optimal reverse Stackelberg strategy whose graph intersects with $ W^{2d} $ only at the  point $ (u^{1d},u^{2d},u^{3d}) $. That is,
\begin{center}
	$ W^{2d}\cap \text{Graph}(\gamma^{1*})= (u^{1d},u^{2d},u^{3d})$.
\end{center}

In  what follows  we consider  leader's affine  function mapping of dimension $ m_{1} $
\begin{equation}\label{E6}
	\gamma^{1}: \Omega_{2}\times \Omega_{3} \rightarrow \Omega_{1}
\end{equation}
that satisfies Eq. (\ref{E1}).

The leader's affine function $ \gamma^{1} $ is constructed in such a way that for all possible pair of actions $ (u^{2},u^{3})\in \Omega_{2} \times \Omega_{3} $ there exists   $ u^{1}\in \Omega_{1}  $ such that $ u^{1} = \gamma^{1}(u^{2},u^{3}) $. This can be accomplished by constructing an inverse affine function $ \alpha:\Omega_{1}  \rightarrow\Omega_{2}\times \Omega_{3}  $ whose graph passes through $ (u^{1d},u^{2d},u^{3d}) $.
Towards this, we construct the set of affine relations whose graph passes through $ (u^{1d},u^{2d},u^{3d}) $ as follows.

Denote the set of all affine relations of dimension $ m_{2}+ m_{3} $ in $ \Omega_{2}\times \Omega_{3}$ whose graph passes through $ (u^{1d},u^{2d},u^{3d}) $ by $ \mathcal{A}_{1} $. An element  $ \alpha_{1}\in \mathcal{A}_{1}$ satisfies
\[ \text{Graph}(\alpha_{1}) \cap W^{2d}= (u^{1d},u^{2d},u^{3d}).\]

Since $ \alpha_{1}\in \mathcal{A}_{1}$ is an affine relation having full dimension  $ m_{2}+ m_{3} $, for all  $ (u^{2},u^{3})\in \Omega_{2}\times \Omega_{3} $, there exists $ u^{1}\in \Omega_{1} $ such that $ \alpha_{1}(u^{1})= (u^{2},u^{3}) $. Thus, for every  $ \alpha_{1}\in \mathcal{A}_{1}$ we have  $ \alpha_{1}(\Omega_{1}) = \Omega_{2}\times\Omega_{3}$. Then the candidate leader function is characterized by $ \gamma^{1}:= (\alpha_{1})^{-1} $. Next, let us consider the set of affine relations whose graph passes through  $ (u^{1d},u^{2d},u^{3d}) $ and  lie on the supporting hyperplane $ \Pi_{W^{2d}} $. This set is denoted by $ \mathcal{A}^{\Pi_{W^{2d}}}_{1}$ and is defined as

\[\mathcal{A}^{\Pi_{W^{2d}}}_{1}:= \{\alpha_{1}\in \mathcal{A}_{1}: \alpha_{1} \subseteq\Pi_{W^{2d}}\}.\]

Now, we state the following two results from \cite{luenberger1969optimization} that we shall use them in the subsequent analysis.

\begin{lemma}[Geometric  Hahn-Banach Theorem,  Luenberger 1969]\label{T1}
	Let $ K $ be a convex set having a nonempty interior in a real normed linear vector space $ X $. Suppose that $ V $ is a linear variety in $ X $ containing no interior points of $ K $. Then there is a closed hyperplane in $ X $ containing $ V $ but containing no interior points of $ K $.
\end{lemma}

\begin{proof}
	See Theorem 1 in \cite{luenberger1969optimization}.
\end{proof}

\begin{lemma}[Support Theorem, Luenberger 1969]\label{T2}
	If $ x $ is not an interior point of a convex set $ K $ which contains interior points, there is a closed hyperplane $ \Pi $  containing $ x  $ such that $ K $ lies on one side of $ \Pi $.
\end{lemma}

\begin{proof}
	See Theorem 2 in \cite{luenberger1969optimization}.
\end{proof}

Now,  to utilize these concepts in our context we state the following corollaries, that follow from Lemma \ref{T1} and Lemma \ref{T2}, which can be used in the subsequent analysis.

\begin{corollary}\label{T3}
	Assume that $ W^{2d} $ is convex and  locally strictly convex at the point $(u^{1d},u^{2d},u^{3d})$. Let $ \Omega_1=\R^{m_1}, \Omega_2 = \R^{m_2}, \Omega_3 = \R^{m_3}$ and let $ \alpha_1 \in \mathcal{A}_1$ be any affine function  such that $\text{Graph} (\alpha_{1})\cap W^{2d}= (u^{1d},u^{2d},u^{3d}) $. Then $\alpha_{1}$ lies on the hyperplane $\Pi_{W^{2d}}$ supporting $W^{2d}$ at the point $(u^{1d},u^{2d},u^{3d})$.
\end{corollary}
\begin{proof}
	Product of Euclidean spaces $ \Omega_{1}=\R^{m_{1}},  \Omega_{2}=\R^{m_{2}}, \Omega_{3}=\R^{m_{3}} $ is a normed linear space. Hence, $  \Omega_{1}\times \Omega_{2}\times \Omega_{3} $ is a normed linear space and the sublevel set $ W^{2d}\subset \Omega_{1}\times \Omega_{2}\times \Omega_{3} $.
	Moreover, as $ \text{Graph} (\alpha_{1})\cap W^{2d}= (u^{1d},u^{2d},u^{3d}) $, then $ \alpha_{1} $ does not contain interior point of $ W^{2d} $. Then by Lemma \ref{T1} there exists a hyperplane $ \Pi_{W^{2d}} $ through $ (u^{1d},u^{2d},u^{3d}) $ containing affine (linear variety)  $ \alpha_{1} $. Since $  W^{2d} $ is locally strictly convex at  	$ (u^{1d},u^{2d},u^{3d}) $, then  $ W^{2d}\cap \Pi_{ W^{2d}}(u^{1d},u^{2d},u^{3d})= (u^{1d},u^{2d},u^{3d})$. Hence, $ \Pi_{W^{2d}} $ supports  $ W^{2d} $ at $ (u^{1d},u^{2d},u^{3d}) $ follows from the definition of the supporting hyperplane and Lemma \ref{T2}.
\end{proof}

\begin{corollary}\label{T4}
	Assume that $ W^{2d} $ is convex and  locally strictly convex at the point $ (u^{1d},u^{2d},u^{3d}) $. Let $ \Omega_1=\R^{m_1}, \Omega_2 = \R^{m_2}, \Omega_3 = \R^{m_3}$. Then a supporting hyperplane $\Pi_{ W^{2d}}(u^{1d},u^{2d},u^{3d})$  intersects with  $ W^{2d} $ only at the point $ (u^{1d},u^{2d},u^{3d}) $.
\end{corollary}

\begin{proof}
	We have proved  in Corollary \ref{T3} the existence of the supporting hyperplane $\Pi_{ W^{2d}}(u^{1d},u^{2d},u^{3d})$ to $ W^{2d} $ at $(u^{1d},u^{2d},u^{3d})$. Then by the definition of a supporting hyperplane, $(u^{1d},u^{2d},u^{3d}) $ is not in an interior point of  $ W^{2d} $.  Since $ W^{2d} $ is locally strictly convex at $(u^{1d},u^{2d},u^{3d})$, applying Lemma \ref{T2} we see that  $\Pi_{ W^{2d}}$ supports $ W^{2d} $ only at $(u^{1d},u^{2d},u^{3d})$.
\end{proof}

In the following theorem we prove the existence of leader's affine reverse Stackelberg strategy provided the assumptions are satisfied. To consider the trilevel problem in its entirety (Eqs. (\ref{Ea1})-(\ref{E2})) we make the following assumption.

\begin{assumption}\label{Assum:3}
	If the mapping (\ref{E6}) is announced by the leader as an optimal reverse Stackelberg strategy to the followers, then the middle level player can find a strategy $ \gamma^{2}\in \Gamma^{2} $ that satisfies Eq. (\ref{E3}).
\end{assumption}	

This assumption states that in the trilevel hierarchical Stackelberg game, for each action chosen by the leader, the middle level player always has a room to respond optimally in the required direction. The conditions for which this assumption is satisfied are given in Theorem \ref{T16}. But first we shall show in the next arguments that the leader can achieve his/her desired equilibrium solution if he/she announces $ \gamma^{1*} $ as optimal reverse Stackelberg strategy.

\begin{theorem}\label{T41}
	Suppose that  $ W^{2d} $ is  convex and locally strictly convex at the point $ (u^{1d},u^{2d},u^{3d}) $, Assumption 3 is satisfied, $ J_{2}(u^{1},u^{2},u^{3}) $ is Fr\'{e}chet differentiable at $ (u^{1d},u^{2d},u^{3d}) $ and $ \Omega_{1}= \R^{m_{1}}, \Omega_{2}= \R^{m_{2}}, \Omega_{3}= \R^{m_{3}}$. If $ \nabla_{u^{1}}J_{2}(u^{1d},u^{2d},u^{3d})\neq 0 $, then, there exists leader's affine map $ \gamma^{1}: \Omega_{2}\times\Omega_{3} \rightarrow \Omega_{1}$ that realizes the desired equilibrium $ (u^{1d},u^{2d},u^{3d}) $.
\end{theorem}

\begin{proof}
	Since the mapping $ \gamma^{1}$ belongs to $\Gamma^{1}$, it satisfies Eq. (\ref{E1}).  In view of Assumption \ref{Assum:3} the problem faced by the leader is to find $ \gamma^{1} $ that satisfies Eq. (\ref{E3}). From Corollary \ref{T3} and Corollary \ref{T4}, it follows that there exists an affine mapping
	\begin{equation}\label{Ea7}
		\alpha_{1}^{\Pi_{W^{2d}}}\in \mathcal{A}_{1}^{\Pi_{W^{2d}}}
	\end{equation} such that
	\begin{equation*}\label{Eb7}
		\text{Graph}(\alpha_{1}^{\Pi_{W^{2d}}}) \cap W^{2d}= (u^{1d},u^{2d},u^{3d}).
	\end{equation*}
	Now, we need to show that $\alpha_{1}^{\Pi_{W^{2d}}} (\Omega_{1}) = \Omega_{2}\times \Omega_{3}$. Since $ J_{2} $ is Fr\'{e}chet differentiable at $ (u^{1d},u^{2d},u^{3d}) $ the normal vector to $ W^{2d} $ at $ (u^{1d},u^{2d},u^{3d}) $ exists, is unique and  equal to $ \nabla J_{2}(u^{1d},u^{2d},u^{3d}) $, in which case its supporting hyperplane $ \Pi_{ W^{2d}}(u^{1d},u^{2d},u^{3d}) $ is given by		
	\begin{equation}\label{E7}
		\begin{split}
			\langle\nabla_{u^{1}}J_{2}(u^{1d},u^{2d},u^{3d}),u^{1}-u^{1d}\rangle+ \langle\nabla_{u^{2}}J_{2}(u^{1d},u^{2d},u^{3d}),\\ u^{2}-u^{2d}\rangle+\langle\nabla_{u^{3}}J_{2}(u^{1d},u^{2d},u^{3d}),u^{3}-u^{3d}\rangle=0.	
		\end{split}
	\end{equation}
	If $ \nabla_{u^{1}}J_{2}(u^{1d},u^{2d},u^{3d})\neq 0 $, then the normal vector defining the hyperplane \\
    $ \Pi_{W^{2d}}(u^{1d},u^{2d},u^{3d}) $ is not orthogonal to the decision space $ \Omega_{1}$, that means,
	\begin{center}
		$\text{Proj}_{\Omega_{1}}^{{\Pi_{W^{2d}}}(u^{1d},u^{2d},u^{3d})}\neq \{0\} $.
	\end{center}
	Then it follows that the hyperplane is not orthogonal to $\{0\}^{m_{1}}\times \Omega_{2}\times\Omega_{3}$. Thus,
        \[\text{Proj}_{\Omega_{2}\times\Omega_{3}}^{\Pi_{W^{2d}}(u^{1d},u^{2d},u^{3d})} = \Omega_{2}\times\Omega_{3}.\]
	
	Hence, for all $(u^{2},u^{3})\in \Omega_{2}\times\Omega_{3}$ there exists  $ u^{1}\in \Omega_{1}$ such that $ (u^{1},u^{2},u^{3})\in \Pi_{W^{2d}}(u^{1d},u^{2d},u^{3d})$. 	Thus, the affine map $\alpha_{1}^{\Pi_{W^{2d}}}$ whose existence is verified in Eq. (\ref{Ea7}) is full dimensional and fulfills the condition that
	\begin{center}
		$\alpha_{1}^{\Pi_{W^{2d}}} (\Omega_{1}) = \Omega_{2}\times \Omega_{3}$.
	\end{center}
	The  leader's strategy given by $ \gamma^{1}:= (\alpha_{1}^{\Pi_{W^{2d}}})^{-1} $ restricts the optimization of the middle level player to set of points determined by the map $ \gamma^{1} $. Since graph of $ \gamma^{1} $ intersects the level set $ W^{2d} $ only at $ (u^{1d},u^{2d},u^{3d}) $, the minimum of $ J_{2}(\gamma^{1*},u^{2},u^{3}) $ over $ \Omega_{2}\times \Omega_{3} $ is obtained at $ (\gamma^{1*}(u^{2d},u^{3d}),u^{2d},u^{3d}) $.

\end{proof}

\begin{remark}
	If $ \nabla_{u^{1}}J_{2}(u^{1d},u^{2d},u^{3d})= 0$, then  $ J_{2} $ is not sensitive to the change of leader's control variable $ u^{1}$. In this case, the leader can not directly influence the decision of the  middle level player.	 On the other hand, if $ \nabla_{u^{1}}J_{2}(u^{1d},u^{2d},u^{3d})\neq 0$ and $ \nabla_{u^{1}}J_{3}(u^{1d},u^{2d},u^{3d})= 0$	the leader can still achieve the desired equilibrium under the conditions in Theorem \ref{T16}.
\end{remark}

\begin{theorem}\label{T5}
	Suppose that $ W^{2d} $ is  convex and is locally strictly convex at the point $ (u^{1d},u^{2d},u^{3d}) $, Assumption 3 is satisfied, $ J_{2}(u^{1},u^{2},u^{3}) $ is Fr\'{e}chet differentiable at $ (u^{1d},u^{2d},u^{3d}) $ and $ \Omega_{1}= \R^{m_{1}}, \Omega_{2}= \R^{m_{2}}, \Omega_{3}= \R^{m_{3}} $. If the leader's affine map   $\  \gamma^{1}: \Omega_{2}\times\Omega_{3} \rightarrow \Omega_{1}$ realizes the desired equilibrium $ (u^{1d},u^{2d},u^{3d}) $, then
	\[ \nabla_{u^{1}}J_{2}(u^{1d},u^{2d},u^{3d})\neq 0.\]
\end{theorem}	

\begin{proof}
	Since $ J_{2} $ is differentiable at $ (u^{1d},u^{2d},u^{3d}) $, the normal vector to $ W^{2d} $ at $ (u^{1d},u^{2d},u^{3d}) $ exists, is unique and is equal to $ \nabla J_{2}(u^{1d},u^{2d},u^{3d}) $ in which case the supporting hyperplane $ \Pi_{ W^{2d}}(u^{1d},u^{2d},u^{3d}) $ is given by Eq. (\ref{E7}).
	
	The proof is carried out by contrapositive arguments. Suppose that,
	\[ \nabla_{u^{1}}J_{2}(u^{1d},u^{2d},u^{3d})=0.\]
	Then it follows from Eq. (\ref{E7}) that
	\[ \langle\nabla_{u^{2}}J_{2}(u^{1d},u^{2d},u^{3d}),u^{2}-u^{2d}\rangle+\langle\nabla_{u^{3}}J_{2}(u^{1d},u^{2d},u^{3d}),u^{3}-u^{3d}\rangle=0.\]
	Then the normal vector defining the hyperplane $ \Pi_{W^{2d}}(u^{1d},u^{2d},u^{3d})$ is
	\[ \mathbf{v} = (0,\nabla_{u^{2}}J_{2}(u^{1d},u^{2d},u^{3d}),\nabla_{u^{3}}J_{2}(u^{1d},u^{2d},u^{3d})).\]
    Since $W^{2d}$ is locally strictly convex at $(u^{1d},u^{2d},u^{3d})$, the hyperplane \\
    $\Pi_{W^{2d}}(u^{1d},u^{2d},u^{3d})$ is orthogonal to $\{0\}^{m_{1}}\times \Omega_{2}\times\Omega_{3}  $. That means,
	\[\Omega_{2}\times\Omega_{3} \subsetneq \text{Proj}_{\Omega_{2}\times\Omega_{3}}^{\Pi_{W^{2d}}(u^{1d},u^{2d},u^{3d})}.\]
	
	Hence, $ \Pi_{W^{2d}}(u^{1d},u^{2d},u^{3d})$ does  not include any element $ (u^{1},u^{2},u^{3})\in$ $\Omega_{1}\times (\Omega_{2}\times\Omega_{3}\backslash \{(u^{2d},u^{3d})\}) $, which implies that $ \alpha_{1}^{\Pi_{W^{2d}}}(\Omega_{1})\subsetneq \Omega_{2}\times\Omega_{3} $.
	That means, $ \gamma^{1} = (\alpha_{1}^{\Pi_{W^{2d}}})^{-1} $ does not hold. As a result the affine map $ \gamma^{1} $ can not be defined.

\end{proof}

\begin{proposition}\label{T7}
	Suppose that all the conditions of Theorem \ref{T41} are satisfied to guarantee the existence of $\gamma^1$. Then for all $ (u^{1},u^{2},u^{3})\in \text{Graph}(\gamma^{1}) $, we have
	\begin{equation}\label{ET1}
		J_{2}(u^{1},u^{2},u^{3})\geq J_{2}(u^{1d},u^{2d},u^{3d}),
	\end{equation}
	where equality can only be achieved at the  equilibrium point $ (u^{1d},u^{2d},u^{3d}) $, which is desired by the leader.
\end{proposition}

\begin{proof}
	Suppose $ (u^{1},u^{2},u^{3})\in \text{Graph}(\gamma^{1})$, then by definition we have  $ (u^{1},u^{2},u^{3})\in \Pi_{W^{2d}}$. Since $ \Pi_{W^{2d}} $ supports $ W^{2d} $ at $ (u^{1d},u^{2d},u^{3d}) $, $ W^{2d} $ lies on one side of $ \Pi_{W^{2d}}$. Hence for all $  (u^{1},u^{2},u^{3})\in \Pi_{W^{2d}} $, we have
	\begin{equation}\label{E9}
		J_{2}(u^{1},u^{2},u^{3}) \geq  J_{2}(u^{1d},u^{2d},u^{3d})
	\end{equation}
	
	Since $\text{Graph}(\gamma^{1})$ lies on $\Pi_{W^{2d}} $, the result follows as well in $ \text{Graph}(\gamma^{1}) $.
	Then, it follows from Eqs. (\ref{E5}) and (\ref{E9}) that equality is achieved at the equilibrium point $ (u^{1d},u^{2d},u^{3d}) $.	
\end{proof}

Therefore, it is the best (rational) interest  of the middle level player  to find a reverse Stackelberg strategy $ \gamma^{2} $ to induce the follower's decision to the desired  value $ u^{3}= u^{3d} $.

Next, we analyze the existence  of a reverse Stackelberg strategy $ \gamma^{2}\in\Gamma^{2} $ which satisfies Assumption \ref{Assum:3}. The following proposition guarantees the existence of a supporting hyperplane to the sublevel set $ W^{3d} $ of $ J_{3}$ at $ (u^{1d},u^{2d},u^{3d})  $.

\begin{proposition}\label{T8}
	If $ W^{3d} $ is  convex and  locally strictly convex at $ (u^{1d},u^{2d},u^{3d}) $, $ J_{3}(u^{1},u^{2},u^{3}) $ is Fr\'{e}chet differentiable at $ (u^{1d},u^{2d},u^{3d}) $ and $ \Omega_{1}= \R^{m_{1}}, \Omega_{2}= \R^{m_{2}}, \Omega_{3}= \R^{m_{3}} $ then there exists a supporting hyperplane $ \Pi_{W^{3d}} $ that intersects $ W^{3d} $  uniquely at the point $ (u^{1d},u^{2d},u^{3d}) $.
\end{proposition}

\begin{proof}
	The result follows automatically  from Corollary \ref{T3} and Corollary \ref{T4}.
\end{proof}

In consideration of the case where $ \nabla_{u^{1}} J_{3}(u^{1d},u^{2d},u^{3d})\neq 0 $ we then analyze the influence of $ \gamma^{1} $ on $ J_{3} $.

Since $ J_{3}(u^{1},u^{2},u^{3}) $ is Fr\'{e}chet differentiable at $ (u^{1d},u^{2d},u^{3d}) $ the equation of the hyperplane $ \Pi_{W^{3d}} $ can be written as 	
\begin{equation}\label{E11}
	\begin{split}
		\langle\nabla_{u^{1}}J_{3}(u^{1d},u^{2d},u^{3d}),u^{1}-u^{1d}\rangle+ \langle\nabla_{u^{2}}J_{3}(u^{1d},u^{2d},u^{3d}),u^{2}-u^{2d}\rangle\\ +\langle\nabla_{u^{3}}J_{3}(u^{1d},u^{2d},u^{3d}),u^{3}-u^{3d}\rangle=0.
	\end{split}
\end{equation}

By virtue of Theorem \ref{T41}, if $ \nabla_{u^{1}}J_{3}(u^{1d},u^{2d},u^{3d})\neq 0 $, there exists an affine leader function
\begin{equation}\label{ET3}
	\gamma^{1'}: \Omega_{2}\times\Omega_{3}\rightarrow \Omega_{1}
\end{equation}
that lies on the supporting hyperplane $ \Pi_{W^{3d}} $ to the sublevel set $ W^{3d} $.

Note that both the affine relations $\gamma^{1}$ and $\gamma^{1'}$ intersect the sublevel sets $ W^{2d} $ and $ W^{3d} $  at the point $ (u^{1d},u^{2d},u^{3d}) $ and lie on the supporting hyperplanes  $ \Pi_{W^{2d}} $ and $ \Pi_{W^{3d}} $ respectively.

The two supporting hyperplanes $ \Pi_{W^{2d}} $ and $ \Pi_{W^{3d}} $ passing through $ (u^{1d},u^{2d},u^{3d}) $ have common points.  Therefore, there exists  a set containing points that satisfy both Eqs. (\ref{E7}) and (\ref{E11}). Define this set $ \Phi $ as

\begin{equation}\label{E13}
	\Phi = \{(u^{1},u^{2},u^{3})\in \Omega_{1}\times\Omega_{2}\times\Omega_{3}~\vert~~(u^{1},u^{2},u^{3}) \  \text{satisfies}  \ \ \text{Eqs.} (\ref{E7})\ \ \text{and}\ \ (\ref{E11}) \}.
\end{equation}

   However, the main results of this article follow regardless of the value of  $ \nabla_{u^{1}} J_{3}(u^{1d},u^{2d},u^{3d}) $.
That means, the leader can acquire the desired solution even when  $ \nabla_{u^{1}} J_{3}(u^{1d},u^{2d},u^{3d})= 0 $ by declaring a strategy $ \gamma^{1} $ in conformity with Eq. (\ref{E4}).

\begin{proposition}\label{T9}
	Suppose that all the conditions of Theorem \ref{T41} are satisfied.	Then for $ (u^{1},u^{2},u^{3})\in \Phi $, we have
	\begin{equation}\label{ET2}
		J_{3}(u^{1},u^{2},u^{3}) \geq  J_{3}(u^{1d},u^{2d},u^{3d}),
	\end{equation}
	where equality can only be achieved at the desired equilibrium point $ (u^{1d},u^{2d},u^{3d}) $.
\end{proposition}

\begin{proof}
	If $ \Phi $ is a singleton, then	$ \Phi = \{(u^{1d},u^{2d},u^{3d})\} $ and the result follows automatically. Suppose $ \Phi $ is not a singleton.
	Let $ (u^{1},u^{2},u^{3})\in \Phi $ such that $ (u^{1},u^{2},u^{3})\neq (u^{1d},u^{2d},u^{3d}) $. Then by definition we have  $ (u^{1},u^{2},u^{3})\in \Pi_{W^{3d}} $. Since $ \Pi_{W^{3d}} $ supports $ W^{3d} $ at $ (u^{1d},u^{2d},u^{3d}) $, $ W^{3d} $ lies on one side of $ \Pi_{W^{3d}} $. Hence for all $  (u^{1},u^{2},u^{3})\in \Pi_{W^{3d}} $, we have
	\begin{equation}\label{E14}
		J_{3}(u^{1},u^{2},u^{3}) \geq  J_{3}(u^{1d},u^{2d},u^{3d}).
	\end{equation}
	Then, it follows from Eqs. (\ref{E10}) and (\ref{E14}) that equality is achieved at the equilibrium point $ (u^{1d},u^{2d},u^{3d}) $.	
\end{proof}	

Therefore, if the middle level player can find a strategy $ \gamma^{2}\in \Gamma^{2} $, corresponding to leader's strategy $ \gamma^{1} $, to constrain the  follower's decision space to $ \Phi $, the  optimal performance available to the  follower can only be achieved at the point $ (u^{1d},u^{2d},u^{3d}) $. That also means that the problem faced by the middle level player is to find $ \gamma^{2} $ that satisfy Eq. (\ref{E4}).

Substituting the fixed leader's optimal reverse Stackelberg strategy  $ \gamma^{1*} $, we denote $ J_{3}(\gamma^{1*},u^{2},u^{3}) $ by $ \bar{J}_{3}(u^{2},u^{3}) $. For such a case, the sublevel set of $ J_{3} $ at $ (u^{1d},u^{2d},u^{3d}) $ is denoted and defined as	
\begin{equation}\label{E19}
	\bar{W}^{3d}= \{(u^{2},u^{3})\in \Omega_{2}\times\Omega_{3}~\vert~~\bar{J}_{3}(u^{2},u^{3})\leq \bar{J}_{3}(u^{2d},u^{3d})\}.
\end{equation}

\begin{remark}
	Since $ \bar{W}^{3d} $ is the outcome of $ W^{3d} $ for fixed  affine leader's strategy $ \gamma^{1*} $, it is  convex and locally strictly convex at $  (u^{2d},u^{3d}) $.
\end{remark}

Substituting the fixed leader's optimal affine reverse Stackelberg strategy \\
$ \gamma^{1*}(u^{1},u^{2},u^{3}) $ in Eq. (\ref{E11}) and using Assumption \ref{Assum:2} we have

\begin{equation}\label{ET6}
	\langle\nabla_{u^{2}}\bar{J}_{3}(u^{2d},u^{3d}), u^{2}-u^{2d}\rangle+\langle\nabla_{u^{3}}\bar{J}_{3}(u^{2d},u^{3d}),u^{3}-u^{3d}\rangle=0
\end{equation}	
which is a hyperplane orthogonal to the decision space $ \Omega_{2}\times \Omega_{3} $. The middle level player can retain his/her desired solution $ (u^{2d},u^{3d})$ by an  affine strategy  $ \gamma^{2}:\Omega_{3}\rightarrow \Omega_{2} $ 	whose graph intersects $ \bar{W}^{3d} $ only at $ (u^{2d},u^{3d})$. This idea will be proved in Theorem \ref{T16} below.

Denote the set of all affine relations of dimension $  m_{3} $ in $ \Omega_{3}$ whose graph passes through $ (u^{2d},u^{3d}) $ by $ \mathcal{A}_{2} $. Then an element  $ \alpha_{2}\in \mathcal{A}_{2}$ satisfies
\[ \text{Graph}(\alpha_{2}) \cap \bar{W}^{3d}= (u^{2d},u^{3d}).\]
Since $ \alpha_{2}\in \mathcal{A}_{2}$ is an affine relation having full dimension  $  m_{3} $, for all  $ u^{3} \in  \Omega_{3} $, there exists $ u^{2}\in \Omega_{2} $ such that $ \alpha_{2}(u^{2})= u^{3} $. Thus, for every  $ \alpha_{2}\in \mathcal{A}_{2}$ we have  $ \alpha_{2}(\Omega_{2}) = \Omega_{3}$. Then the candidate for the leader's function is characterized by $ \gamma^{2}:= (\alpha_{2})^{-1} $. Next, let us consider the set of affine relations whose graph passes through  $ (u^{2d},u^{3d}) $ and  lie on the supporting hyperplane $ \Pi_{\bar{W}^{3d}} $. This set is denoted by $ \mathcal{A}^{\Pi_{\bar{W}^{3d}}}_{2}$ and is defined as

\[\mathcal{A}^{\Pi_{\bar{W}^{3d}}}_{2}:= \{\alpha_{2}\in \mathcal{A}_{2}: \alpha_{2} \subseteq\Pi_{\bar{W}^{3d}}\}.\]

\begin{theorem}\label{T16}
	Suppose that   $\bar{W}^{3d} $ is  convex and locally strictly convex at the point $ (u^{2d},u^{3d}) $, $ \bar{J}_{3}(u^{2},u^{3})$ is Fr\'{e}chet differentiable at $ (u^{2d},u^{3d}) $ and $  \Omega_{2}= \R^{m_{2}}, \Omega_{3}= \R^{m_{3}} $. Under the announced leader's optimal reverse Stackelberg strategy $ \gamma^{1*} $, if $ \nabla_{u^{2}}\bar{J}_{3}(u^{2d},u^{3d})\neq 0 $, there exists  middle level player's optimal  strategy $ \gamma^{2*}: \Omega_{3} \rightarrow \Omega_{2}$ that satisfies Eq. (\ref{E4}).    	
\end{theorem}

\begin{proof}
	Since the mapping $ \gamma^{2} $ is assumed to satisfy Eq. (\ref{E2}) by Assumption \ref{Assum:2}, the problem faced by the middle level player is to find $ \gamma^{2} $ that satisfies Eq. (\ref{E4}). From Corollary \ref{T3} and Corollary \ref{T4}, it follows that there exists an affine mapping
	
	\begin{equation}\label{Ea3}
		\alpha_{2}^{\Pi_{\bar{W}^{3d}}}\in \mathcal{A}_{2}^{\Pi_{\bar{W}^{3d}}}
	\end{equation} such that
	\begin{equation*}\label{Eb3}
		\text{Graph}(\alpha_{2}^{\bar{W}^{3d}}) \cap \bar{W}^{3d}= (u^{2d},u^{3d}).
	\end{equation*}
	
	Now, we need to show that $\alpha_{2}^{\Pi_{\bar{W}^{3d}}} (\Omega_{2}) =  \Omega_{3}$. Since $ \bar{J}_{3} $ is Fr\'{e}chet differentiable at $ (u^{2d},u^{3d}) $ the normal vector to $ \bar{W}^{3d} $ at $ (u^{2d},u^{3d}) $ exists, is unique and  equal to $ \nabla \bar{J}_{2}(u^{2d},u^{3d}) $, in which case its supporting hyperplane $ \Pi_{ \bar{W}^{3d}}(u^{2d},u^{3d}) $ is determined by the relation in	Eq. (\ref{ET6}).	
	
	If $ \nabla_{u^{2}}\bar{J}_{2}(u^{2d},u^{3d})\neq 0 $, then the normal vector defining the hyperplane $ \Pi_{\bar{W}^{3d}}(u^{2d},u^{3d}) $ is not orthogonal to the decision space $ \Omega_{2}$, that means,
	\[\text{Proj}_{\Omega_{2}}^{{\Pi_{\bar{W}^{3d}}}(u^{2d},u^{3d})}\neq \{0\}.\]
	Then it follows that the hyperplane is not orthogonal to $\{0\}^{m_{2}}\times\Omega_{3}$. Thus,
    \[\text{Proj}_{\Omega_{3}}^{\Pi_{\bar{W}^{3d}}(u^{2d},u^{3d})} = \Omega_{3}.\]
	
	Hence, for all 	$u^{3}\in \Omega_{3}$ there exists  $ u^{2}\in \Omega_{2}$ such that $ (u^{2},u^{3})\in \Pi_{\bar{W}^{3d}}(u^{2d},u^{3d})$.
	Thus, the affine map $\alpha_{2}^{\Pi_{\bar{W}^{3d}}}$ whose existence is verified in Eq. (\ref{Ea3}) is full dimensional and fulfills the relation:
	\[\alpha_{2}^{\Pi_{\bar{W}^{3d}}} (\Omega_{2}) =  \Omega_{3}.\]
	
	The  middle level player's strategy given by $ \gamma^{2}:= (\alpha_{2}^{\Pi_{\bar{W}^{3d}}})^{-1} $ restricts the optimization of the follower to set of points determined by the map $ \gamma^{1} $ and $ \gamma^{2} $. Since the graph of $ \gamma^{2} $ intersects the level set $ \bar{W}^{3d} $ only at $ (u^{2d},u^{3d}) $, the minimum of $ \bar{J}_{3}(\gamma^{1*},\gamma^{2},u^{3}) $ is obtained at $ u^{3d} $.

\end{proof}

The  supporting hyperplanes $ \Pi_{W^{2d}} $ and $ \Pi_{\bar{W}^{3d}}  $  through the point $ (u^{1d},u^{2d},u^{3d}) $ have common points.
Therefore, there exists  a set containing points that satisfy both Eqs. (\ref{E7}) and (\ref{ET6}). Define this set $ \Psi $ as	
\begin{equation}\label{E26}
	\Psi = \{(u^{1d},u^{2},u^{3})\in \Omega_{1}\times\Omega_{2}\times\Omega_{3}~\vert~~(u^{1d},u^{2},u^{3}) \  \text{satisfies} \ \text{Eqs.} (\ref{E7}) \ \text{and} \ (\ref{ET6}) \ \}.
\end{equation}

\begin{proposition}\label{T15}
	Suppose that the conditions of Theorems \ref{T41} and \ref{T5} are satisfied so that the  existence of $ \gamma^{1} $ and $ \gamma^{2} $ are guarantied.	Then for $ (u^{1d},u^{2},u^{3})\in \Psi$, we have
	\begin{equation}\label{ET5}
		J_{3}(u^{1d},u^{2},u^{3}) \geq  J_{3}(u^{1d},u^{2d},u^{3d}),
	\end{equation}
	where equality can only be achieved at the desired equilibrium point $ (u^{1d},u^{2d},u^{3d}) $.
\end{proposition}
\begin{proof}
	Follows a similar argument as in the proof of Proposition \ref{T9}.

\end{proof}

\begin{lemma}\label{T12}
	Suppose that the conditions of Theorems \ref{T41} and \ref{T5} are satisfied. Then for optimal reverse Stackelberg strategies  $ \gamma^{1*} $ and  $ \gamma^{2*} $, the sets $  \Phi $ and  $ \Psi$ are equal.  	
\end{lemma}
\begin{proof}
	Suppose that $ (u^{1},u^{2},u^{3})\in \Phi $, then Eq. (\ref{E7}) follows automatically and Eq. (\ref{ET6}) follows  by  Eq. (\ref{E1}).
	Therefore, $ \Phi \subseteq\Psi $.
	
	Conversely if $ (u^{1d},u^{2},u^{3})\in \Psi $, then we have Eq. (\ref{ET6}) which is equal to Eq. (\ref{E11}) for fixed $ \gamma^{1*} $.  And Eq. (\ref{E7})  follows automatically.
	Hence,  $ \Psi \subseteq\Phi $.

\end{proof}

The optimal reverse Stackelberg strategies $ \gamma^{1*} $ and $ \gamma^{2*} $   lie on the supporting hyperplanes  $ \Pi_{W^{2d}} $ and $ \Pi_{\bar{W}^{3d}}  $ respectively.  Therefore, results of Proposition \ref{T9}, Proposition \ref{T15} and Lemma \ref{T12} holds for the set of points determined by the affine reverse Stackelberg  strategies $ \gamma^{1*} $ and $ \gamma^{2*} $.

We conclude the discussion  about the existence of  optimal reverse Stackelberg strategies $ \gamma^{1*} $ and  $ \gamma^{2*} $  with the following theorem.
\begin{theorem}\label{T16d}
	Suppose that $ W^{2d} $ and  $\bar{W}^{3d} $ are  convex and locally strictly convex at the point $(u^{1d},u^{2d},u^{3d})$ and $(u^{2d},u^{3d})$, respectively, $J_{2}(u^{1},u^{2},u^{3})$ and $\bar{J}_{3}(u^{2},u^{3})$ are Fr\'{e}chet differentiable at $(u^{1d},u^{2d},u^{3d})$ and $(u^{2d},u^{3d})$ respectively and $ \Omega_{1} = \R^{m_{1}}, \Omega_{2} = \R^{m_{2}}, \Omega_{3}= \R^{m_{3}}$. If  $\nabla_{u^{1}}J_{2}(u^{1d},u^{2d},u^{3d})\neq 0$, $\nabla_{u^{2}}\bar{J}_{3}(u^{2d},u^{3d})\neq 0$, then there exist optimal reverse Stackelberg strategies $\gamma^{1*} $ and $ \gamma^{2*}$, that guarantee the achievement of the equilibrium solution $(u^{1d},u^{2d},u^{3d})$.    	
\end{theorem}

\begin{proof}
	The proof follows from  Theorem \ref{T41} and Theorem \ref{T16}.

\end{proof}

\subsection{Nonconvex sublevel sets }

In this Section, we consider nonconvex sublevel sets $ W^{2d} $ and $ W^{3d} $. By taking the convex hulls of  $ W^{2d} $ and $ W^{3d} $, we develop  conditions under which results for the convex case can be applied.

When the sublevel sets  $ W^{2d} $ and $ W^{3d} $ are allowed to be nonconvex, the desired solution $ (u^{1d},u^{2d},u^{3d}) $ of the leader may not be at the boundary point of the sublevel sets. So, the requirement  	Graph$( \alpha_{1}) \cap W^{2d}= (u^{1d},u^{2d},u^{3d}) $ fails to hold. Consequently, the results in Theorem \ref{T41} and Theorem \ref{T16} can not be applied directly. In the following propositions, we present conditions under which results of Theorem \ref{T41} and Theorem \ref{T16} can be applied  to nonconvex sublevel sets  $ W^{2d} $ and $ W^{3d} $. In what follows, we denote the convex hull of a set $ A $ by $ conv(A) $.

\begin{proposition}\label{T17}
	Let  $ conv(W^{2d}) $ be locally strictly convex at $ (u^{1d},u^{2d},u^{3d}) $ and assume that  $ J_{2}(u^{1},u^{2},u^{3}) $ is Fr\'{e}chet differentiable at $ (u^{1d},u^{2d},u^{3d}) $ and $ \Omega_{1}= \R^{m_{1}}, \Omega_{2}= \R^{m_{2}}, \Omega_{3}= \R^{m_{3}}$. If $ (u^{1d},u^{2d},u^{3d}) $ is an exposed point of $ conv(W^{2d}) $ and $ \nabla_{u^{1}}J_{2}(u^{1d},u^{2d},u^{3d})\neq 0 $, then  there exists leader's affine map $ \gamma^{1}: \Omega_{2}\times\Omega_{3} \rightarrow \Omega_{1}$ that realizes the desired equilibrium $ (u^{1d},u^{2d},u^{3d}) $.
\end{proposition}
\begin{proof}
	As $ (u^{1d},u^{2d},u^{3d}) $ is an exposed point of $ conv(W^{2d}) $, the existence of a  hyperplane   $ \Pi_{conv (W^{2d})}$ that supports $ conv(W^{2d}) $ at $ (u^{1d},u^{2d},u^{3d}) $ follows by definition. Furthermore, we have
    \[ \Pi_{conv (W^{2d})}(u^{1d},u^{2d},u^{3d})\cap conv(W^{2d})=(u^{1d},u^{2d},u^{3d}).\]
Noting that the desired equilibrium point $(u^{1d},u^{2d},u^{3d})$ is in $W^{2d}$ and $W^{2d}\subset conv(W^{2d}) $, the hyperplane $ \Pi_{conv (W^{2d})}(u^{1d},u^{2d},u^{3d})$ supports $ W^{2d} $ at $ (u^{1d},u^{2d},u^{3d}) $ as well.
	Define an affine map $ \alpha_{1}^{\Pi_{conv (W^{2d})}}: \Omega_{1}\rightarrow\Omega_{2}\times \Omega_{3} $  that lies on the supporting hyperplane $  \Pi_{conv (W^{2d})}(u^{1d},u^{2d},u^{3d})$. Now, we need to check whether this map satisfies  $\alpha_{1}^{\Pi_{conv(W^{2d})}} (\Omega_{1}) = \Omega_{2}\times \Omega_{3}$. But this follows from the proof of Theorem 1  replacing  $ W^{2d} $ by $ conv(W^{2d})$.

\end{proof}

\begin{proposition}\label{T18}
	Suppose that   $conv(\bar{W}^{3d}) $ is locally strictly convex at $ (u^{1d},u^{2d},u^{3d}) $, $ \bar{J}_{3}(u^{2},u^{3})  $ is differentiable at $ (u^{2d},u^{3d}) $ and $  \Omega_{2}= \R^{m_{2}}, \Omega_{3}= \R^{m_{3}} $. Under the announced leader's optimal reverse Stackelberg strategy $ \gamma^{1*} $, if  $ (u^{1d},u^{2d},u^{3d}) $ is an exposed point of $ conv(W^{2d}) $ and $ \nabla_{u^{2}}\bar{J}_{3}(u^{2d},u^{3d})\neq 0 $, then, there exists  the middle level player's optimal  strategy $ \gamma^{2*}: \Omega_{3} \rightarrow \Omega_{2}$ that satisfies Eq. (\ref{E4}).    	
\end{proposition}

\begin{proof}
	Since	$ (u^{1d},u^{2d},u^{3d}) $ is assumed to be an exposed point of $conv(\bar{W}^{3d})$, the existence of a supporting hyperplane $\Pi_{\bar{W}^{3d} }(u^{1d},u^{2d},u^{3d})$ that supports $ conv(\bar{W}^{3d}) $ at $ (u^{1d},u^{2d},u^{3d}) $ follows from its definition. Furthermore, we have
 \[\Pi_{conv (\bar{W}^{3d})}(u^{1d},u^{2d},u^{3d})\cap conv(\bar{W}^{3d})=(u^{1d},u^{2d},u^{3d}).\]
 Noting that the desired equilibrium $ (u^{1d},u^{2d},u^{3d})$ is in $\bar{W}^{3d} $ and $ \bar{W}^{3d}\subset conv(\bar{W}^{3d}) $, the hyperplane $ \Pi_{conv (\bar{W}^{3d})}(u^{1d},u^{2d},u^{3d})$ supports $\bar{W}^{3d} $ at $ (u^{1d},u^{2d},u^{3d}) $ as well.
 	Define an affine map $ \alpha_{2}^{\Pi_{conv (\bar{W}^{3d})}}: \Omega_{2}\rightarrow \Omega_{3} $  that lie on the supporting hyperplane $  \Pi_{conv (\bar{W}^{3d})}(u^{1d},u^{2d},u^{3d})$. Now,we  check whether this map satisfies  $\alpha_{2}^{\Pi_{conv(\bar{W}^{3d})}} (\Omega_{2}) =  \Omega_{3}$.

	Again the assertion that   $\alpha_{2}^{\Pi_{conv(\bar{W}^{3d})}} (\Omega_{2}) = \Omega_{3}$ follows from Theorem \ref{T16}  replacing  $ \bar{W}^{3d} $ by $ conv(\bar{W}^{3d})$.

\end{proof}

\begin{remark} The idea described above can be extended to any finite $n$-level reverse Stackelberg problem that has similar structure as in problems \ref{Ea1} -- \ref{E2}. In deed, once the leaders optimal affine maping $\gamma^{1*}$ is constructed, since affine strategy of the leader ($ 1^{\textrm{st}} $-level player) preserves convexity of sublevel sets of the followers, the structure of the objective functions of the remaining $2^{\textrm{nd}},\ldots, (n-1)^{\textrm{th}}, n^{\textrm{th}}$ level players remains unaffected. Next, the $ 2^{\textrm{nd}}$-level player  in the hierarchy  constructs reverse Stackelberg strategy to preserve his/her team optimal solution induced by the announced strategy of the leader. This process continues up to the $ (n-1)^{\textrm{th}} $-level player in the vertical hierarchy which constructs a reverse Stackelberg strategy to obtain the team optimal solution induced by the strategy of the $ (n-2)^{\textrm{th}} $-level player in the hierarchy. This task can be achieved by repeating a procedure used for finding optimal reverse Stackelberg strategy of the middle level player in trilevel  games $ (n-2) $ number of times sequentially
\end{remark}

\section{Characterization of Optimal Reverse Stackelberg Strategies}\label{sect:characSt}
	
	In the first part of this Section, we provide illustrative examples to support the foregoing  existence results. Reverse Stackelberg strategies for the leader are constructed based on Fr\'{e}chet derivatives  of $ J_{2} $ and $ J_{3} $ at the desired solution point $ (u^{1d},u^{2d},u^{3d}) $. The solution method which is also discussed in \cite{mizukami1989constructions} is not all inclusive. It  results in a single leader's affine reverse Stackelberg strategy, where in reality a number of strategies can be constructed. In the second part of the Section, we develop a more general  method to construct multiple reverse Stackelberg strategies for the leader.
	
	\subsection{A single solution case}\label{subsect:char1}
	
	In this Subsection an affine strategy $ \gamma^{1}: \Omega_{2}\times\Omega_{3}\rightarrow\Omega_{1} $ that yields the desired equilibrium point $(u^{1d},u^{2d},u^{3d})$ is characterized. Following the presentation in  \cite{mizukami1989constructions} we assume that the decision spaces $ \Omega_{1},\Omega_{2},\Omega_{3}$ are unconstrained Hilbert spaces, $ J_{2},J_{3} $  are Fr\'{e}chet differentiable and   convex.  The determination of leader's function $ \gamma^{1} $ in the finite dimensional case (with $ \Omega_{1}=\R^{m_1},\Omega_{2}=\R^{m_2}$ and $\Omega_{3}=\R^{m_3}$), which is given by
	\begin{equation}\label{Ex1}
	u^{1}:= \gamma^{1}(u^{2},u^{3})= u^{1d}- Q_{1}(u^{2}-u^{2d})- Q_{2}(u^{3}-u^{3d}),
	\end{equation}
satisfies Eq. (\ref{E1}) automatically, and is	reduces to a computation of $ m_{1}\times m_{2} $ matrix $ Q_{1} $ and $ m_{1}\times m_{3} $ matrix $ Q_{2} $. In the finite dimensional Hilbert space the linear operators  $ Q_{1} $ and $ Q_{2} $ satisfy
	\begin{equation}\label{Ex2}
		Q_{1}^{*}\nabla_{u_{1}}J_{2}(u^{1d},u^{2d},u^{3d}) =  \nabla_{u_{2}}J_{2}(u^{1d},u^{2d},u^{3d}),
\end{equation}

\begin{equation}\label{Ex3}	
	 Q_{2}^{*}\nabla_{u_{1}}J_{2}(u^{1d},u^{2d},u^{3d}) =  \nabla_{u_{3}}J_{2}(u^{1d},u^{2d},u^{3d}),
\end{equation}
	 where $ Q_{i}^{*}, i=1,2 $ denote the adjoint of $ Q_{i}, i=1,2 $. For finite dimensional spaces we use the fact that $ Q_{i}^{*}= Q_{i}^{T}$.

	In this case the leader's affine reverse Stackelberg strategy is
	
	\begin{equation}\label{EEx1}
	\begin{split}
	\gamma^{1}(u^{2},u^{3})= u^{1d}- \left( \begin{matrix}\frac{\nabla_{u_{1}}J_{2}(u^{1d},u^{2d},u^{3d})}{\langle \nabla_{u_{1}}J_{2}(u^{1d},u^{2d},u^{3d}),\nabla_{u_{1}}J_{2}(u^{1d},u^{2d},u^{3d})\rangle}\end{matrix}\right)\qquad \qquad\\ \times\left( \begin{matrix}
	\langle\nabla_{u_{2}}J_{2}(u^{1d},u^{2d},u^{3d}),u^{2}-u^{2d}\rangle
	\end{matrix}\right)\\ -  \left( \begin{matrix}\frac{\nabla_{u_{1}}J_{2}(u^{1d},u^{2d},u^{3d})}{\langle \nabla_{u_{1}}J_{2}(u^{1d},u^{2d},u^{3d}),\nabla_{u_{1}}J_{2}(u^{1d},u^{2d},u^{3d})\rangle}\end{matrix}\right)\qquad \qquad\\ \times\left( \begin{matrix}
		\langle\nabla_{u_{3}}J_{2}(u^{1d},u^{2d},u^{3d}),u^{3}-u^{3d}\rangle
		\end{matrix}\right).
	\end{split}	
	\end{equation}
	Substituting Eq. (\ref{Ex1}) into the supporting hyperplane in   Eq. (\ref{E11}) and simplifying we get
	\begin{equation}\label{Ex4}
	\begin{split}
	\langle\nabla_{u^{2}}J_{3}(u^{1d},u^{2d},u^{3d}) - Q_{1}^{*}(\nabla_{u^{1}}J_{3}(u^{1d},u^{2d},u^{3d})), u^{2}- u^{2d}\rangle \ \ \ \ \ \ \ \ \ \ \   \\
	+ \langle\nabla_{u^{3}}J_{3}(u^{1d},u^{2d},u^{3d}) - Q_{2}^{*}(\nabla_{u^{1}}J_{3}(u^{1d},u^{2d},u^{3d})), u^{3}- u^{3d}\rangle = 0.
	\end{split}
	\end{equation}
	
	The resulting Eq. (\ref{Ex4}) represents a hyperplane that is orthogonal to the product space  \ $ \Omega_{2}\times\Omega_{3} $. As mentioned in \cite{mizukami1989constructions}, if
     \[\nabla_{u^{2}}J_{3}(u^{1d},u^{2d},u^{3d}) - Q_{1}^{*}(\nabla_{u^{1}}J_{3}(u^{1d},u^{2d},u^{3d})) \neq 0,\]
then there exists a bounded linear operator $ Q_{3}^{*}$ that satisfies
	\begin{equation}\label{Ex5}
	\begin{split}
	Q_{3}^{*} [\nabla_{u^{2}}J_{3}(u^{1d},u^{2d},u^{3d}) - Q_{1}^{*}(\nabla_{u^{1}}J_{3}(u^{1d},u^{2d},u^{3d}))] \\= [\nabla_{u^{3}}J_{3}(u^{1d},u^{2d},u^{3d}) - Q_{2}^{*}(\nabla_{u^{1}}J_{3}(u^{1d},u^{2d},u^{3d}))],
	\end{split}	
	\end{equation}
	such that	an affine function
	
	\begin{equation}\label{Ex6}
	u^{2}=  \gamma^{2}(u^{3})   = u^{2d}- Q_{3}(u^{3}- u^{3d}),
	\end{equation}
	 lies on the hyperplane determined by Eq. (\ref{Ex4}).

	Denote expressions in Eq. (\ref{Ex5}) as
	 \begin{equation*}
	 \begin{split}
	 	\bar{u}_{321}=\nabla_{u^{2}}J_{3}(u^{1d},u^{2d},u^{3d}) - Q_{1}^{*}(\nabla_{u^{1}}J_{3}(u^{1d},u^{2d},u^{3d})),  \\
	 \bar{u}_{331}= \nabla_{u^{3}}J_{3}(u^{1d},u^{2d},u^{3d}) - Q_{2}^{*}(\nabla_{u^{1}}J_{3}(u^{1d},u^{2d},u^{3d})).
	 \end{split}
	\end{equation*}
 Then Eq. (\ref{Ex6}) can be re-written as
 \begin{equation}\label{EEx2}
 	\gamma^{2}(u^{3})   = u^{2d}- \frac{\bar{u}_{321}}{\langle \bar{u}_{321},\bar{u}_{321}\rangle}\langle\bar{u}_{331},u^{3}-u^{3d}\rangle.
 \end{equation}

That means, given a trilevel game with differentiable cost functions of followers, the top level player can construct optimal reverse Stackelberg strategy $\gamma^{1*}$ using the expression given in Eq. (\ref{EEx1}). The strategy declared by the leader induces the middle level player to construct $\gamma^{2*}$ using the expressions given in  Eq. (\ref{EEx2}). These strategies are set up in such a way that the leader can achieve his/her desired equilibrium.
 \begin{example}
 	Consider the trilevel problem with objective functions of each hierarchical level is respectively given by
 	\begin{equation*}
 	\begin{split}
 	&J_{1}(u^{1},u^{2},u^{3}) = (u^{1}-2)^{2}+(u^{2}-1)^{2}+(u^{3}-3)^{2},\\
 	&J_{2}(u^{1},u^{2},u^{3}) = (u^{1}-1)^{2}+(u^{2})^{2}+(u^{3})^{2},\\
 	&J_{3}(u^{1},u^{2},u^{3}) = (u^{1})^{2}+(u^{2}-2)^{2}+(u^{3})^{2}.
 	\end{split}
 	\end{equation*}
 \end{example}
	
	We obtain  team optimal solution  $(u^{1d},u^{2d},u^{3d})= (2,1,3) $ by minimizing $ J_{1} $ over $\R^3$. So, the leader's optimal reverse Stackelberg strategy  $ \gamma^{1*} $ that modifies follower's objective functions so as to  reach his/her desired equilibrium must satisfy
	\begin{equation*}
	\begin{split}	
	& \argmin_{u^{2},u^{3}} J_{2}(\gamma^{1*}(u^{2},u^{3}),u^{2},u^{3})=(1,3),\\	
	&\gamma^{1*}(1,3)= 2.
	\end{split}
	\end{equation*}
	Since $ J_{2} $ is differentiable, $\nabla J_{2}(u^{1},u^{2},u^{3})= (2(u^{1}-1),2u^{2},2u^{3})^{\top}$
	and  at the desired equilibrium point   $(2,1,3), \nabla J_{2}(2,1,3)= (2,2,6)^{\top}  $. From Eqs. (\ref{Ex2}) and (\ref{Ex3}) at the team optimal solution $ Q_{1}= 1, Q_{2}= 3 $. In this case, the optimal reverse Stackelberg strategy of the leader is
	\[ \gamma^{1*}= -u^{2} - 3u^{3} +12.\]

	When the leader announces his/her optimal reverse Stackelberg strategy $ \gamma^{1*} $ we verify that Eqs. (\ref{E1}) and (\ref{E2}) are satisfied.	
	\begin{equation*}
		\begin{split}	
			& \argmin_{u^{2},u^{3}} [(-u^{2} - 3u^{3} +12)^{2}+(u^{2})^{2}+(u^{3})^{2}]=(1,3),\\	
			&\gamma^{1*}(1,3)= 2.
		\end{split}
	\end{equation*}
	In reaction to announcement of  $ \gamma^{1*} $ by the leader, the middle level player constructs $ \gamma^{2*}$ that satisfies Eqs. (\ref{E3}) and (\ref{E4}) by solving :
	\begin{equation*}
	\begin{split}	
		& \argmin_{u^{3}} J_{3}(\gamma^{1*}(u^{2},u^{3}),\gamma^{2*}(u^{3}),u^{3})= 3,\\	
		&\gamma^{2*}(3)=1.
	\end{split}
\end{equation*}
Since $ \nabla J_{3}(u^{1},u^{2},u^{3})= (2u^{1},2(u^{2}-2),2u^{3})^\top  $, the Fr\'{e}chet derivative at the desired equilibrium is $ \nabla J_{3}(2,1,3)= (4,-2,6)^\top  $. It follows from Eq. (\ref{Ex5}) that $ Q_{3}= 1 $ and from Eq. (\ref{EEx1}) that optimal  strategy of  the middle level player is  \begin{center}
		$ \gamma^{2*}= -u^{3}+4 $.
	\end{center}
	
\begin{example}
	Consider the general quadratic trilevel dynamic game with the objective functions for each $i = 1,2,3$ is given by
	\begin{equation}\label{EX21}
	J_{i}(u^{1},u^{2},u^{3})= \sum_{j,k=1,2,3}\langle u^{j},A^{i}_{jk}u^{k} \rangle + \sum_{k=1,2,3}\langle u^{k},l^{i}_{k}\rangle,
	\end{equation}
 where $ k\geq j $ and  $ A^{j}_{ik} $ are linear bounded operators, $ A^{j}_{ii} $  are strongly positive, $ A^{i}_{ki}=0 $ for $ i\neq k $, $ l^{i}_{k}\in\Omega_{k} $ are known and the conditions on $ A^i_{jk} $ makes the Hessian matrix non negative definite for all players.
\end{example}
For $ i=1,2,3 $,  the cost function of each of the players becomes
\begin{equation*}
\begin{split}
J_{i}= \langle u^{1},A^{i}_{11}u^{1} \rangle + \langle  u^{1},A^{i}_{12}u^{2} \rangle + \langle  u^{2},A^{i}_{22}u^{2}\rangle + \langle  u^{1},A^{i}_{13}u^{3} \rangle\\
+ \langle  u^{2},A^{i}_{23}u^{3} \rangle + \langle u^{3},A^{i}_{33}u^{3} \rangle + \langle u^{1},l_{1}^{i} \rangle + \langle u^{2},l_{2}^{i} \rangle + \langle u^{3},l_{3}^{i} \rangle.
\end{split}
\end{equation*}
 The Fr\'{e}chet derivatives of $ J_{i}$ with respect to $ u^{i},i=1,2,3$ are
	\begin{align}
    	\nabla_{u^{1}}J_{i}(u^{1},u^{2},u^{3})= 2A_{11}^{i}u^{1}+A_{12}^{i}u^{2}+A_{13}^{i}u^{3}+l_{1}^{i},\\
    	\nabla_{u^{2}}J_{i}(u^{1},u^{2},u^{3})= A_{12}^{i}u^{1}+2A_{22}^{i}u^{2}+A_{23}^{i}u^{3}+l_{2}^{i},\\
    	\nabla_{u^{3}}J_{i}(u^{1},u^{2},u^{3})= A_{13}^{i}u^{1}+A_{23}^{i}u^{2}+2A_{33}^{i}u^{3}+l_{3}^{i}.
	\end{align}
	The desired equilibrium for the leader  $( u^{1d},u^{2d},u^{3d}) $ is found as a solution of the system
		\begin{equation}\label{g8}
		\left(
			\begin{matrix}
			2A_{11}^{i} \quad A_{12}^{i} \quad   A_{13}^{i} \\
			A_{12}^{i} \quad 2A_{22}^{i} \quad  A_{23}^{i}\\
			A_{13}^{i} \quad  A_{23}^{i} \quad 2A_{33}^{i}
			\end{matrix}
		\right) \left(
		\begin{matrix}
	 u^{1}\\
	 u^{2}\\
	 u^{3}
		\end{matrix}
		\right)=\left(
		\begin{matrix}
		-l^{i}_{1}\\
		-l^{i}_{2}\\
	    -l^{i}_{3}		
		\end{matrix}
		\right),
		\end{equation}
	for each $i = 1,2,3$.

	Since $ J_{1} $ is convex, minimization of $ J_{1} $ with respect to $(u^{1},u^{2},u^{3})$ provides the desired team optimal solution, say $(u^{1d},u^{2d},u^{3d}) $. The  sets $ W^{2d} $ and $ W^{3d} $, given by
		\begin{equation*}
	W^{2d} = \{(u^{1},u^{2},u^{3})\in \Omega_{1}\times \Omega_{2}\times \Omega_{3}: J_{2}(u^{1},u^{2},u^{3})\leq J_{2}(u^{1d},u^{2d},u^{3d})\}
	\end{equation*}	
	\begin{equation*}
	W^{3d} = \{(u^{1},u^{2},u^{3})\in \Omega_{1}\times \Omega_{2}\times \Omega_{3}: J_{3}(u^{1},u^{2},u^{3})\leq J_{3}(u^{1d},u^{2d},u^{3d})\}
	\end{equation*}
	are convex (as each is a sublevel set of a convex function).

	Supporting hyperplanes of $ W^{2d} $ and $ W^{3d} $ at  $( u^{1d},u^{2d},u^{3d}) $ are given by
		\begin{equation*}
	\begin{split}
	\langle\nabla_{u^{1}}J_{i}(u^{1d},u^{2d},u^{3d}),u^{1}-u^{1d}\rangle+ \langle\nabla_{u^{2}}J_{i}(u^{1d},u^{2d},u^{3d}),\\ u^{2}-u^{2d}\rangle+\langle\nabla_{u^{3}}J_{i}(u^{1d},u^{2d},u^{3d}),u^{3}-u^{3d}\rangle=0,	
	\end{split}
	\end{equation*}
	where $ i=2,3 $. To construct  optimal reverse Stackelberg strategy of the leader,  gradients of $ J_{2} $ and $ J_{3} $ at the team optimal solution  $ (u^{1d},u^{2d},u^{3d}) $ are
	\begin{equation*}
	\begin{split}
		\nabla_{u^{1}}J_{i}(u^{1d},u^{2d},u^{3d})= 2A_{1d}^{i}u^{1d}+A_{12}^{i}u^{2d}+A_{13}^{i}u^{3d}+l_{1}^{i},\\
		\nabla_{u^{2}}J_{i}(u^{1d},u^{2d},u^{3d})= A_{12}^{i}u^{1d}+2A_{22}^{i}u^{2d}+A_{23}^{i}u^{3d}+l_{2}^{i},\\
		\nabla_{u^{3}}J_{i}(u^{1d},u^{2d},u^{3d})= A_{13}^{i}u^{1d}+A_{23}^{i}u^{2d}+2A_{33}^{i}u^{3d}+l_{3}^{i}.
	\end{split}
	\end{equation*}
	
	Then according to Eqs. (\ref{Ex1})--(\ref{EEx1}), the  optimal reverse Stackelberg strategy of the leader is

\begin{equation}\label{g81}
	 \gamma^{1*}(u^{2},u^{3}) = u^{1d} - Q_{1}(u^{2}-u^{2d})-Q_{3}(u^{3}-u^{3d}),
\end{equation}
where \[ Q_{1}=\frac{\langle2A_{11}^{1}u^{1d}+A_{12}^{1}u^{2d}+A_{13}^{1}u^{3d}+l_{1}^{1},A_{12}^{1}u^{1d}+2A_{22}^{1}u^{2d}+A_{23}^{1}u^{3d}+l_{1}^{1} \rangle}{\langle 2A_{11}^{1}u^{1d}+A_{12}^{1}u^{2d}+A_{13}^{1}u^{3d}+l_{1}^{1},2A_{11}^{1}u^{1d}+A_{12}^{1}u^{2d}+A_{13}^{1}u^{3d}+l_{1}^{1}\rangle}, \]
\[ Q_{2}=\frac{\langle2A_{11}^{1}u^{1d}+A_{12}^{1}u^{2d}+A_{13}^{1}u^{3d}+l_{1}^{1},A_{13}^{1}u^{1d}+A_{23}^{1}u^{2d}+2A_{33}^{1}u^{3d}+l_{1}^{1} \rangle}{\langle 2A_{11}^{1}u^{1d}+A_{12}^{1}u^{2d}+A_{13}^{1}u^{3d}+l_{1}^{1},2A_{11}^{1}u^{1d}+A_{12}^{1}u^{2d}+A_{13}^{1}u^{3d}+l_{1}^{1}\rangle}.  \]

Substituting the fixed leader's optimal affine reverse Stackelberg strategy
 $ u^{1}=\gamma^{1*}(u^{2},u^{3}) $ in $ J_{3}(u^{1},u^{2},u^{3}) $
the supporting hyperplane for
\begin{equation*}
\bar{W}^{3d}= \{(u^{2},u^{3})\in \Omega_{2}\times\Omega_{3}~\vert~~\bar{J}_{3}(u^{2},u^{3})\leq \bar{J}_{3}(u^{2d},u^{3d})\},
\end{equation*}
is given by	
\begin{equation*}
\langle\nabla_{u^{2}}\bar{J}_{3}(u^{2d},u^{3d}), u^{2}-u^{2d}\rangle+\langle\nabla_{u^{3}}\bar{J}_{3}(u^{2d},u^{3d}),u^{3}-u^{3d}\rangle=0.
\end{equation*}	

The middle level player reacts to the leader's reverse Stackelberg  strategy Eq. (\ref{g81}) by constructing $  \gamma^{2*} $ based on Eq. (\ref{EEx2}) as
\begin{equation*}
\gamma^{2*}(u^{3})   = u^{2d}- \frac{\bar{u}_{321}}{\langle \bar{u}_{321},\bar{u}_{321}\rangle}\langle\bar{u}_{331},u^{3}-u^{3d}\rangle,
\end{equation*}
provided that  $ \bar{u}_{332}\neq 0 $ and 	
 \begin{equation*}
	\begin{split}
		\bar{u}_{321}=	 A_{12}^{3}u^{1d}+2A_{22}^{3}u^{2d}+A_{23}^{3}u^{3d}+l_{2}^{3}-Q_{1}(2A_{1d}^{3}u^{1d}+A_{12}^{3}u^{2d}+A_{13}^{3}u^{3d}+l_{1}^{3}),\\		
	\bar{u}_{331}=	A_{13}^{3}u^{1d}+A_{23}^{3}u^{2d}+2A_{33}^{3}u^{3d}+l_{3}^{3}-Q_{2}(A_{13}^{3}u^{1d}+A_{23}^{3}u^{2d}+2A_{33}^{3}u^{3d}+l_{3}^{3}).
		\end{split}	
\end{equation*}

 \subsection{General Characterization}\label{subsect:char2}

 	In Subsection \ref{subsect:char1}, we briefly discussed construction of affine reverse Stackelberg strategies for trilevel games based on the Fr\'{e}chet derivatives of objective functions of followers. However, as illustrated in the following trilevel game, it is possible that the leader can have an unlimited number of strategies to achieve his/her team optimal solution. So, in this section we present a method to construct multiple optimal affine reverse Stackelberg strategies of the leader.     	
 	
 	\begin{example}\label{Ex:3}
 		Consider a three person trilevel game with decision spaces $ u^{1}\in \R^{2},u^{2}\in \R $ and $ u^{3}\in \R $ characterized by
 		\begin{equation}\label{ge1}
 		\begin{split}
 		&J_{1}(u^{1},u^{2},u^{3}) = (u^{1}_{1})^{2}+(u^{1}_{2})^{2}+(u^{2})^{2}+(u^{3})^{2}+5u^{1}_{1}+3u^{1}_{2}+u^{2} + u^{3},\\
 		&J_{2}(u^{1},u^{2},u^{3}) = (u^{1}_{1})^{2}+(u^{1}_{2})^{2}+(u^{2})^{2}+(u^{3})^{2}+ 3u^{2},\\
 		&J_{3}(u^{1},u^{2},u^{3}) = (u^{1}_{1})^{2}+(u^{1}_{2})^{2}+(u^{2})^{2}+(u^{3})^{2},
 		\end{split}
 		\end{equation}
 	where $ u^{1}=\left(
 	\begin{matrix}
 		u^{1}_{1} \\
 		u^{1}_{2}
 	\end{matrix}\right) $.
 	\end{example}
 	
 	In this example the leader's desired equilibrium  is $ (\frac{-5}{2},\frac{-3}{2},\frac{-1}{2},\frac{-1}{2}) $ and the reverse Stackelberg strategy of the leader is	
 	\begin{equation}\label{ge2}
 	\gamma^{1}=\left(
 	\begin{matrix}
 	\frac{-5}{2} \\
 	\frac{-3}{2}
 	\end{matrix}
 	\right) + \left(
 	\begin{matrix}
 	\frac{1}{5}(2-3t_{1})\\
 	t_{1}		
 	\end{matrix}
 	\right)(u^{2}+\frac{1}{2})+ \left(
 	\begin{matrix}
 	-\frac{1}{5}(1+3t_{2})\\
 	t_{2}		
 	\end{matrix}
 	\right)(u^{3}+\frac{1}{2}),
 	\end{equation}
 	where $ t_{1}$ and $ t_{2}$ are free parameters.
 Therefore, there are  infinitely many reverse Stackelberg strategies for this problem. Note that, the characterization of strategies that is used in Subsection \ref{subsect:char1} only yields a singe solution if we apply it also to this example.

 Here, results in  \cite{groot2016optimal} for bi-level games were extended to derive  general characterizations of leader's affine reverse Stackelberg strategies for trilevel games.  We propose the following leader's function.

 \begin{equation}\label{g1}
 \gamma^{1}= u^{1d}- R_{1}s_{1}- R_{1}'s_{2}
 \end{equation}
 where $ R_{1}, R_{1}',s_{1}$ and  $ s_{2} $ are matrices of appropriate dimensions  satisfying

 \begin{equation}\label{g2}
 \left(\begin{matrix}
 u^{1}\\
 u^{2}
 \end{matrix}\right)=\left(\begin{matrix}
 u^{1d}\\
 u^{2d}
 \end{matrix}\right)+\left(\begin{matrix}
 R_{1}\\
 R_{2}
 \end{matrix}\right)s_{1},
 \end{equation}
 \begin{equation}\label{g3}
 \left(\begin{matrix}
 u^{1}\\
 u^{3}
 \end{matrix}\right)=\left(\begin{matrix}
 u^{1d}\\
 u^{3d}
 \end{matrix}\right)+\left(\begin{matrix}
 R_{1}'\\
 R_{3}
 \end{matrix}\right)s_{2},
 \end{equation}
 provided that $ R_{2} $ and $ R_{3} $ are nonsingular. Solving for $ s_{1} $ and $ s_{2} $ from Eqs. (\ref{g2}) and (\ref{g3}), Eq. (\ref{g1}) can be written as
 	\begin{equation}\label{g4}
 u^{1}:= \gamma^{1}(u^{2},u^{3})= u^{1d}- R_{1}R_{2}^{-1}(u^{2}-u^{2d})- R_{1}'R_{3}^{-1}(u^{3}-u^{3d}).
 \end{equation}

 In order to explore the general characterization of $ \gamma^{1} $ in Eqs. (\ref{g1})-(\ref{g4}), set $ R^{1}=[R_{1}^\top \quad R_{2}^\top], R^{1}\in \R^{(m_{1}+m_{2})\times m_{2}}$ and $ R^{2}=[R_{1}'^\top \quad R_{3}^\top], R^{2}\in \R^{(m_{1}+m_{3})\times m_{3}}$.

  \begin{lemma}\label{lemma:4}
  	If the leader function $ \gamma^{1} $ given in Eq. (\ref{g4}) is optimal for $  R^{1}=[R_{1}^\top\quad R_{2}^\top] $ and  $ R^{2}=[R_{1}'^\top\quad R_{3}^\top] $, the following conditions hold.
  	\begin{enumerate}
  		\item The matrices  $  R^{1} $ and  $ R^{2}$ satisfy
  		\begin{eqnarray}
  	&	[\nabla_{u^{1}} J_{2}(u^{1d},u^{2d},u^{3d}) \quad \nabla_{u^{2}} J_{2}(u^{1d},u^{2d},u^{3d})]^\top R^{1}=0,\label{gm1}  \\  & [\nabla_{u^{1}} J_{2}(u^{1d},u^{2d},u^{3d}) \quad \nabla_{u^{3}} J_{2}(u^{1d},u^{2d},u^{3d})]^\top R^{2}=0 \label{gm2}.
  		\end{eqnarray}
  		\item The columns of $ R_{2} $  and $ R_{3} $ should form bases for $ \Omega_{2} $ and $ \Omega_{3} $ respectively, i.e., $ R_{2} $  and $ R_{3} $ are of full rank  matrices of size $ m_{2} $ and $ m_{3} $ respectively.
  	\end{enumerate}
  \end{lemma}

\begin{proof}
	\begin{enumerate}
   \item Substituting the proposed strategy Eq. (\ref{g4}) for the leader in Eq. (\ref{E7}) we have 	
    \begin{equation*}\label{gp1}
    	\begin{split}
    	\langle\nabla_{u^{1}}J_{2}(u^{1d},u^{2d},u^{3d}), R_{1}R_{2}^{-1}(u^{2}-u^{2d})
    	+ R_{1}'R_{3}^{-1}(u^{3}-u^{3d})\rangle\\+ \langle\nabla_{u^{2}}J_{2}(u^{1d},u^{2d},u^{3d}), u^{2}-u^{2d}\rangle  \qquad \qquad \qquad\\
    	+\langle\nabla_{u^{3}}J_{2}(u^{1d},u^{2d},u^{3d}),u^{3}-u^{3d}\rangle=0.\\
    	\end{split}
    \end{equation*}	
Then it follows that	
	\begin{equation*}
	\begin{split}
	\langle\nabla_{u^{1}}J_{2}(u^{1d},u^{2d},u^{3d}), R_{1}R_{2}^{-1}(u^{2}-u^{2d})\rangle \qquad \qquad \qquad\\+ \langle\nabla_{u^{1}}J_{2}(u^{1d},u^{2d},u^{3d}),R_{1}'R_{3}^{-1}(u^{3}-u^{3d})\rangle \qquad\qquad\\+ \langle\nabla_{u^{2}}J_{2}(u^{1d},u^{2d},u^{3d}), u^{2}-u^{2d}\rangle \qquad \qquad \\+\langle\nabla_{u^{3}}J_{2}(u^{1d},u^{2d},u^{3d}),u^{3}-u^{3d}\rangle=0.\\
	\end{split}
	\end{equation*}	
	Thus
	\begin{equation*}
	\begin{split}
		\langle (R_{1}R_{2}^{-1})^\top\nabla_{u^{1}}J_{2}(u^{1d},u^{2d},u^{3d}), (u^{2}-u^{2d})\rangle \qquad \qquad\qquad \\+ \langle (R_{1}'R_{3}^{-1})^\top\nabla_{u^{1}}J_{2}(u^{1d},u^{2d},u^{3d}),(u^{3}-u^{3d})\rangle \qquad  \\+ \langle\nabla_{u^{2}}J_{2}(u^{1d},u^{2d},u^{3d}), u^{2}-u^{2d}\rangle \qquad \qquad \\+\langle\nabla_{u^{3}}J_{2}(u^{1d},u^{2d},u^{3d}),u^{3}-u^{3d}\rangle\qquad\\		
	=\langle(R_{1}R_{2}^{-1})^\top\nabla_{u^{1}}J_{2}(u^{1d},u^{2d},u^{3d})\qquad\qquad\qquad \qquad\qquad\\+\nabla_{u^{2}}J_{2}(u^{1d},u^{2d},u^{3d}), (u^{2}-u^{2d})\rangle\qquad \qquad\qquad\\+ \langle (R_{1}'R_{3}^{-1})^\top\nabla_{u^{1}}J_{2}(u^{1d},u^{2d},u^{3d})\qquad \qquad\qquad\\+\nabla_{u^{3}}J_{2}(u^{1d},u^{2d},u^{3d}),(u^{3}-u^{3d})\rangle=0
	\end{split}
	\end{equation*}
	from which the first statement of the lemma follows.
	
	\item According to Theorem \ref{T41}, for an  affine leader's function $ \gamma^{1} $ of the form Eq. (\ref{g3}) to be admissible, $ R_{2} $ and $ R_{3} $ must be basis of  $ \Omega_{2} $ and $ \Omega_{3} $ respectively.
\end{enumerate}
\end{proof}
\begin{lemma}
	If there exists an optimal affine mapping $ \gamma^{1} $ characterized by Eq. (\ref{g4}), one can select $ R_{2}= I_{m_{2}}$ and $ R_{3}= I_{m_{3}}$, without any loss of generality.
\end{lemma}
\begin{proof}
The selection $ R_{2}= I_{m_{2}}$ and $ R_{3}= I_{m_{3}}$ satisfies full dimensionality requirement mentioned in Lemma \ref{lemma:4}. Now, we prove under a selection $ R_{2}= I_{m_{2}}$ and $ R_{3}= I_{m_{3}}$ whether the affine strategy given in Eq. (\ref{g4}) lies on the supporting hyperplane determined by Eq. (\ref{E7}). This simply  follows from the proof of Lemma \ref{lemma:4} by substituting $ I_{m_{2}} $ and $ I_{m_{3}} $ for $ R_{2}$ and $ R_{3}$ respectively.

\end{proof}

When   $ R_{2}= I_{m_{2}}$ and $ R_{3}= I_{m_{3}}$, the following equations can be derived from Eqs. (\ref{gm1}) and (\ref{gm2}).
\begin{eqnarray}
 \nabla_{u^{1}} J_{2}(u^{1d},u^{2d},u^{3d})R_{1} = -\nabla_{u^{2}} J_{2}(u^{1d},u^{2d},u^{3d})^\top I_{m_{2}}\label{g5},\\
\nabla_{u^{1}} J_{2}(u^{1d},u^{2d},u^{3d})R_{1}' = -\nabla_{u^{3}} J_{2}(u^{1d},u^{2d},u^{3d})^\top I_{m_{3}} \label{g6}.
\end{eqnarray}

We conclude the general characterization of leader's reverse Stackelberg strategy by the following theorem  that follows from Theorem \ref{T41}, Lemma \ref{T12} and Lemma \ref{lemma:4}. Moreover, equations that help to find matrices $ R_{1}$ and $ R_{1}'$ satisfying Eqs. (\ref{g5}) and (\ref{g6}) under fixed basis matrices  $ R_{2}= I_{m_{2}}$ and $ R_{3}= I_{m_{3}}$ are given.

\begin{theorem}\label{Th:4}
	Let $ \Omega_{1}= \R^{m_{1}}, \Omega_{2}= \R^{m_{2}}, \Omega_{3}= \R^{m_{3}} $ and conditions of Theorem \ref{T41} are satisfied. Then an affine leader's function of the form Eq. (\ref{g4}) exists where the matrices $ R_{1}$ and $ R_{1}'$ belong to affine spaces of the form
	\begin{equation}\label{g7}
	\begin{split}
   \mathcal{R}_{1}=\{R_{1}=R_{11}^{0} + B_{N} t_{1}\},	\\
   \mathcal{R}_{1}'= \{R_{1}'=R_{12}^{0} + B_{N} t_{2}\},	
	\end{split}
	\end{equation}
	where $ R_{11}^{0} $ and $ R_{12}^{0} $ are particular solutions of (\ref{g5}) and (\ref{g6}) respectively and $ B_{N}\in \text{Null}\{\nabla_{u^{1}} J_{2}(u^{1d},u^{2d},u^{3d})\} $.
\end{theorem}

\begin{proof}
Since $ \nabla_{u^{1}} J_{2}(u^{1d},u^{2d},u^{3d})\neq 0 $, Eqs. (\ref{g5}) and (\ref{g6})	can be solved as a system of equalities. Moreover, for each zero element $  \nabla_{u^{1,j}} J_{2}(u^{1d},u^{2d},u^{3d}), j= 1, \ldots, m_{2} $  the corresponding entry $ R^{0}_{1j} $ is free. Therefore, Eq. (\ref{g7}) is a possible solution of Eqs. (\ref{g5}) and (\ref{g6}).
\end{proof}

Now, given the leader's optimal reverse Stackelberg strategy  $ \gamma^{1*} $, the objective functions $ J_{2}(\gamma^{1*},u^{2},u^{3}) $ and $ J_{3}(\gamma^{1*},u^{2},u^{3}) $ become a function of $ u^{2} $ and $ u^{3} $ only, and hence we denote the later by $ \bar{J}_{3}(u^{2},u^{3}) $ for which a corresponding result for two level problem works.

\begin{example}\label{Ex:4}
	In Example \ref{Ex:3} the leader's reverse Stackelberg strategy depends on the free parameters $ t_{1} $ and $ t_{2} $. For example, fixing $ t_{1}=t_{2}=0 $ in Eq. (\ref{ge2}), the leader's reverse Stackelberg strategy becomes
	
	 \begin{equation}\label{g10}
	\left(\begin{matrix}
	u^{1}_{1}\\
	u^{1}_{2}
	\end{matrix}\right)=\left(\begin{matrix}
	\gamma^{1}_{1}\\
	\gamma^{1}_{2}
	\end{matrix}\right)=\left(\begin{matrix}
\frac{1}{5}(2u^{2}-u^{3}-12)\\
	\frac{-3}{2}
	\end{matrix}\right).
	\end{equation}
Substituting Eq. (\ref{g10}) for $ J_{2} $ and $ J_{3} $ in Eq. (\ref{ge1})	we get
\begin{equation}\label{g11}
\begin{split}
&J_{2}(\gamma^{1*},u^{2},u^{3}) = (\frac{1}{5}(2u^{2}-u^{3}-12))^{2}+(	\frac{-3}{2})^{2}+(u^{2})^{2}+(u^{3})^{2}+ 3u^{2},\\
&J_{3}(\gamma^{1*},u^{2},u^{3}) = (\frac{1}{5}(2u^{2}-u^{3}-12))^{2}+(	\frac{-3}{2})^{2}+(u^{2})^{2}+(u^{3})^{2}.
\end{split}
\end{equation}

Now Eq. (\ref{g11}) is a two level problem for which the middle level player's reverse Stackelberg strategy is
\begin{equation*}
\gamma^{2*}=-\frac{1}{2}.
\end{equation*}
\end{example}

\begin{remark}
The choice of  $ t_{1}=t_{2}=0 $ for the free parameters, in Example \ref{Ex:4} above, is random. However, the leader can choose the values of  $ t_{1}$ and $t_{2} $ by considering secondary optimization problem like optimizing the loss incurred, to the leader and punishment to the followers, in case followers deviate from the desired solution.
\end{remark}

\section{Constrained Decision Space}
In this section, we present a constrained trilevel reverse Stackelberg games. Suppose  that $ \Omega_{1}\times\Omega_{2}\times\Omega_{3}\subsetneq \R^{m_{1}}\times \R^{m_{2}} \times\R^{m_{3}} $ and $ \Omega_{1},\Omega_{2},\Omega_{3}$ are compact. In the presence of constraints, the leader solves constrained nonlinear  optimization problem to determine his/her desired equilibrium $ (u^{1d},u^{2d},u^{3d}) $. To utilize the geometric existence theorems presented in  Sections \ref{sect:existance} and \ref{sect:characSt}  for a constrained decision space it should be verified  that supporting hyperplanes $ \Pi_{ W^{2d}}(u^{1d},u^{2d},u^{3d})$ and $\Pi_{ W^{3d}}(u^{1d},u^{2d},u^{3d}) $  containing affine strategies  $ \gamma^{1} $ and $ \gamma^{2} $ belonged to the constrained decision space $ \Omega_{1}\times\Omega_{2}\times\Omega_{3} $ .
 Therefore, existence  conditions of optimal affine reverse Stackelberg strategies presented in  Sections \ref{sect:existance} and \ref{sect:characSt} for  unconstrained version of the trilevel game  provides only necessary conditions in case of a constrained game.

In  \cite{groot2016optimal}, it is reported that proving conditions for existence of leader's reverse Stackelberg strategy  for bilevel games in constrained decision space is a  challenging task. In the same article, the authors were able to filter feasible optimal strategies of a constrained game among infinitely many reverse Stackelberg strategies in unconstrained case of the bilevel problem. We shall extend this concept to the trilevel (and hence to the general $n$-level) case

 Consider a trilevel problem with linear constraints:
\begin{equation}\label{cons:1}
\begin{split}
	& \min_{u^{1}\in \Omega_{1}}	J_{1}(u^{1},u^{2},u^{3}),\\
	&\quad\min_{u^{2}\in \Omega_{2}} J_{2}(u^{1},u^{2},u^{3}),\\
	&\qquad\min_{u^{3}\in \Omega_{3}} J_{3}(u^{1},u^{2},u^{3}),\\
		&\text{subject to:}	\\
	&	A^{1}u^{1}  +A^2 u^2 + A^3u^3 \leq b^{1}, A^{i}\in \R^{k\times m_{i}}, b^{1}\in \R^{k}, \text{ for } i = 1,2,3.
\end{split}
\end{equation}
Here, since the constraints are linear the feasible set is convex.

The first task towards obtaining reverse Stackelberg strategy of the leader is the determination of a desired equilibrium. The leader searches for  a team solution of a constrained nonlinear optimization problem
\begin{equation}\label{cons:2}
	\begin{split}
	&	 \min_{(u^{1},u^{2},u^{3})\in \Omega_{1}\times\Omega_{2}\times\Omega_{3}}	J_{1}(u^{1},u^{2},u^{3}),\\
	&\text{subject to:}	\\
	&	A^{1}u^{1}  +A^2 u^2 + A^3u^3 \leq b^{1}, A^{i}\in \R^{k\times m_{i}}, b^{1}\in \R^{k}, \text{ for } i = 1,2,3.,
\end{split}
\end{equation}
which gives the desired equilibrium $ (u^{1d},u^{2d},u^{3d}) $ under appropriate conditions.

The next task will be to construct feasible affine reverse Stackelberg strategies that satisfy  Eqs.(\ref{E3})-(\ref{E2}). The strategy $ \gamma^{1} $ of the leader is feasible in a constrained decision space if its domain is $ \Omega_{2}\times \Omega_{3} $  and $ \gamma^{1}(\Omega_{2}\times \Omega_{3}) $ is a subset of $ \Omega_{1} $. Similarly, the corresponding strategy $ \gamma^{2} $  is feasible in a constrained decision space if its domain is $ \Omega_{3} $  and $ \gamma^{2}( \Omega_{3}) $ is a subset of $ \Omega_{2} $. The optimal affine strategies constructed for unconstrained version of the problem fulfill all these conditions except the condition for feasibility. Since  the desired equilibrium Eq. (\ref{cons:2}) is in the constrained decision space, strategies that are optimal for unconstrained game are also optimal for a constrained game provided that they are feasible. Therefore, feasible optimal solutions for a constrained game can be obtained from set of affine optimal solutions of the unconstrained version of the problem.

Given a set of  optimal affine  solutions for unconstrained trilevel reverse Stackelberg game,  we verify solvability of a constrained version of the game by determining feasible strategies that are admissible. Thus, leader's optimal strategy $ \gamma^{1} $ in the unconstrained game  is also admissible in a constrained game if it satisfies the condition that for all $ (u^{2},u^{3})\in \Omega_{2}\times\Omega_{3}$ there exists $ u^{1}\in \Omega_{1}$ such that $u^{1}= \gamma^{1}(u^{2},u^{3})$.  Therefore, optimal affine functions discussed in Section \ref{sect:characSt} for unconstrained game satisfying  $ \gamma^{1}(\Omega_{2}\times\Omega_{3})\subseteq \Omega_{1}$ are admissible for a constrained version of the game provided that $ \gamma^{2} $ exists. Furthermore, among set of optimal affine reverse Stackelberg strategies  $ \gamma^{2} $  for unconstrained game, those that satisfy $ \gamma^{2}(\Omega_{3})\subseteq \Omega_{2} $ are also optimal in a constrained game.

In what follows, we present some cases for which the existence of strategies can be verified based on our characterization of multiple optimal strategies in Section \ref{sect:characSt}.
\begin{itemize}
	\item
	If the leader's decision space is constrained by the  linear constraints
	\begin{equation}\label{cns:1}
		A^{1}u^{1}  \leq b^{1}, A^{1}\in \R^{k\times m_{1}}, b^{1}\in \R^{k},
	\end{equation}
	the set of optimal solutions of unconstrained version of the problem can be reduced by adding a constraint
	\begin{equation*}
		A^{1} [u^{1d}- R_{1}(u^{2}-u^{2d})- R_{1}'(u^{3}-u^{3d})]\leq b^{1}, A^{1}\in \R^{k\times m_{1}}, b^{1}\in \R^{k},
	\end{equation*}
	to the expression in Eq. (\ref{g7}).
	\item Suppose $ \Omega_{3}\subset\R^{m_{3}} $ is constrained but  $ \Omega_{1}=\R^{m_{1}},\Omega_{2}=\R^{m_{2}} $ are unconstrained. In this case if $ \gamma^{1}(\Omega_{2}\times\Omega_{3})\subseteq \Omega_{1} $ and $ \gamma^{2}(\Omega_{3})\subseteq \Omega_{2} $, all affine optimal strategies characterized in section \ref{sect:characSt} for unconstrained game are also feasible for constrained version of the game.
	
	\item If decision spaces of all the players are constrained by linear constraints, the feasible region is a convex  polyhedron
	in which affine mappings $ \gamma^{1} $ and $ \gamma^{2} $ preserve convexity. However,  one must verify the satisfaction of the conditions that $\gamma^{1}(\Omega_{2}\times\Omega_{3})\subseteq \Omega_{1} $ and $ \gamma^{2}(\Omega_{3})\subseteq \Omega_{2}$.
\end{itemize}

\section{Conclusion}

In a static multilevel  Stackelberg game, the decision process is sequential from the top $1^\textrm{st}$-level player to $2^\textrm{nd}, \ldots, (n-1)^\textrm{th}, n^\textrm{th}$-level player (the bottom follower). Players at each level optimize their own objective functions which is affected by actions of decision makers at other levels. Reverse Stackelberg strategy of the leader is a mapping from the followers' decision space to the leader's decision space enabling him/her to achieve a desired equilibrium. The results on the existence, and construction of such strategies in the current literature applies only to problems where objective functions of followers are strictly convex and provide only a single strategy to the leader.

This article formulates existence conditions of affine reverse Stackelberg strategy in multilevel static game where the sublevel sets of objective functions of followers' at the desired equilibrium are required to be connected. Moreover, the construction of leader's  multiple optimal  reverse Stackelberg strategies is developed. The attainment of more than one optimal strategy provides an opportunity to consider secondary optimization criteria as well as to solve a constrained game.

However, more research is still needed to address similar type problems with multiple decision entities that are acting according to Nash game at each level of hierarchy. The development of numerical solution techniques is also another area to be worked on. One can also consider development of  nonlinear reverse Stackelberg strategies which are more stable as compared to the affine strategies.

Another important direction for future work is  existence and construction of reverse Stackelberg strategy  for multilevel differential game, in which case the state evolves according to a differential equation and the  performance criteria is integral. In continuous time setting, the strategies of the leader and the middle level player should satisfy causality constraint to be  admissible as opposed to the static case.

%


\bibliographystyle{spmpsci}
\bibliography{TrilevelNEW1}

%
%
%
%

\end{document}